\documentclass[12pt]{amsart}

\usepackage[paperheight=11in,paperwidth=8.5in,left=1in,top=1.25in,right=1in,bottom=1.25in,]{geometry}
\usepackage[linktocpage=false,colorlinks=true,linkcolor=black,citecolor=black,
filecolor=black,urlcolor=black,pagebackref=false,pdfstartpage={1},
pdfstartview={FitH},pdftitle={},pdfauthor={},pdfsubject={},pdfcreator={},
pdfproducer={},pdfkeywords={}]{hyperref}

\usepackage[utf8]{inputenc}
\usepackage{tikz}
\usepackage{amssymb}
\usepackage{rotating}
\usepackage{geometry}
\usepackage{scalerel}
\usetikzlibrary{arrows, arrows.meta, bending, decorations.pathmorphing}
\newcommand{\midarrow}{\tikz \draw[-triangle 60] (0,0) -- +(.05,0);}

\usepackage[lofdepth,lotdepth]{subfig}
\usepackage{comment}
\usepackage{afterpage}
\usepackage{amsfonts}
\usepackage{amsmath}
\allowdisplaybreaks[4]
\usepackage{amssymb}
\usepackage{amsthm}
\usepackage{bbm}
\usepackage[justification=centering]{caption}
\usepackage[capitalize,nameinlink,noabbrev,nosort]{cleveref}
\usepackage{enumerate}
\usepackage{enumitem}
\usepackage{float}
\restylefloat{figure}
\usepackage{kbordermatrix}
\usepackage{youngtab}
\usepackage{psfrag}
\usepackage[boxsize=.3 em]{ytableau}

\usepackage{graphicx}
\usepackage{mathrsfs}
\usepackage{mathtools}
\usepackage[framemethod=tikz,ntheorem,xcolor]{mdframed}
\usepackage{parcolumns}
\usepackage{tabularx}
\usepackage{tikz}\pdfpageattr{/Group <</S /Transparency /I true /CS /DeviceRGB>>}
\usetikzlibrary{arrows,calc,intersections,matrix,positioning,through}
\usepackage{tikz-cd}
\tikzset{commutative diagrams/.cd,every label/.append style = {font = \normalsize}}
\usepackage[all]{xy}
\usepackage{rotating}

\numberwithin{equation}{section}

\newtheorem*{theorem*}{Theorem}
\newtheorem*{corollary*}{Corollary}

\AfterEndEnvironment{thm}{\noindent\ignorespaces}
\newtheorem{theorem}[equation]{Theorem}
\AfterEndEnvironment{theorem}{\noindent\ignorespaces}

\newtheorem{corollary}[equation]{Corollary}
\AfterEndEnvironment{cor}{\noindent\ignorespaces}
\newtheorem{lemma}[equation]{Lemma}
\AfterEndEnvironment{lemma}{\noindent\ignorespaces}

\AfterEndEnvironment{lem}{\noindent\ignorespaces}
\newtheorem{proposition}[equation]{Proposition}
\AfterEndEnvironment{proposition}{\noindent\ignorespaces}

\AfterEndEnvironment{prop}{\noindent\ignorespaces}

\AfterEndEnvironment{conj}{\noindent\ignorespaces}

\AfterEndEnvironment{prob}{\noindent\ignorespaces}

\AfterEndEnvironment{question}{\noindent\ignorespaces}

\theoremstyle{definition}
\newtheorem{defn}[equation]{Definition}
\newtheorem{definition}[equation]{Definition}
\AfterEndEnvironment{defn}{\noindent\ignorespaces}
\AfterEndEnvironment{definition}{\noindent\ignorespaces}
\newtheorem*{pf_no_qed}{Proof}

\AfterEndEnvironment{pf}{\noindent\ignorespaces}
\newtheorem{eg_no_qed}[equation]{Example}

\AfterEndEnvironment{eg}{\noindent\ignorespaces}
\newenvironment{example}[1][]{\begin{eg_no_qed}[#1]\pushQED{\qed}}{\popQED\end{eg_no_qed}}
\AfterEndEnvironment{example}{\noindent\ignorespaces}

\AfterEndEnvironment{rmk}{\noindent\ignorespaces}
\newtheorem{remark}[equation]{Remark}
\AfterEndEnvironment{remark}{\noindent\ignorespaces}

\theoremstyle{remark}

\AfterEndEnvironment{claim}{\noindent\ignorespaces}
\newtheorem*{claimpf_no_qed}{Proof of Claim}

\AfterEndEnvironment{claimpf}{\noindent\ignorespaces}

\newcommand{\Le}{\textup{\protect\scalebox{-1}[1]{L}}}

\font\pipefont=lcircle10
\def\elbow{\smash{\raise3pt\hbox{\pipefont\rlap{\rlap{\char'014}\char'016}}}}

\def\halfelbow{\smash{\raise2pt\hbox{\pipefont\rlap{\rlap{\rlap{\char'015}\phantom{\char'017}}}}}}
\def\cross{\smash{\lower5pt\hbox{\rlap{\vrule height16pt}}\raise3pt\hbox{\rlap{\hskip-8pt \vrule height0.4pt depth0pt width16pt}}}}

\def\Poly{\Pi}

\def\O{\mathcal{O}}
\def\B{\mathcal{B}}

\def\PP{\mathbb{P}}
\def\TT{\mathbb{T}}

\DeclareMathOperator{\convex}{convex}

\DeclareMathOperator{\Web}{Web}

\DeclareMathOperator{\Dr}{Dr}

\newcommand{\R}{\mathbb{R}}
\newcommand{\RR}{\mathbb{R}}
\newcommand{\Z}{\mathbb{Z}}
\newcommand{\K}{\mathbb{K}}

\newcommand{\x}{\mathbf{x}}

\DeclareMathOperator{\val}{val}

\newcommand{\rf}[1]{\hyperref[#1]{(\ref*{#1})}}

\DeclareMathOperator{\Trop}{Trop}

\DeclareMathOperator{\wt}{wt}
\DeclareMathOperator{\Wt}{Wt}

\DeclareMathOperator{\cD}{\mathcal{D}}
\DeclareMathOperator{\cG}{\mathcal{G}}
\DeclareMathOperator{\Trans}{\mathrm{Trans}} 
\DeclareMathOperator{\Span}{\mathrm{Span}}


\title
 {The positive Dressian equals the positive tropical Grassmannian}

\author{David Speyer}
\author{Lauren K.\ Williams}
\address{}
\email{\href{mailto:speyer@umich.edu}{speyer@umich.edu}}
\email{\href{mailto:williams@math.harvard.edu}{williams@math.harvard.edu}}

\thanks{DS was partially supported by NSF grants
 DMS-1855135 and  DMS-1854225.
LW was partially supported by NSF grants DMS-1854316 and DMS-1854512.}

\begin{document}

\begin{abstract}
	The \emph{Dressian} and the \emph{tropical Grassmannian} parameterize 
abstract and realizable tropical linear spaces; but in general,
the Dressian is much larger than the tropical Grassmannian.
There are natural 
	positive notions of both of these spaces -- the 
	\emph{positive Dressian}, and the \emph{positive tropical Grassmannian} 
	(which we introduced 
	 roughly fifteen years
	ago in \cite{troppos}) --  so it is natural
to ask how these two positive spaces compare.  In this paper
we show that the positive Dressian equals the positive
tropical Grassmannian.  Using the connection between the 
positive Dressian and regular positroidal subdivisions of the 
hypersimplex, we use our result to give a new ``tropical'' proof of 
	da Silva's 1987 conjecture (first proved in 2017 by Ardila-Rinc\`{o}n-Williams)
that all positively oriented matroids are realizable. 
We also show that the finest regular positroidal subdivisions 
of the hypersimplex consist of series-parallel matroid polytopes,
and achieve equality in Speyer's 
\emph{$f$-vector theorem}.  Finally we give
an example of a positroidal subdivision of the hypersimplex
which is not regular, and make a connection to the theory of 
tropical hyperplane arrangements.
\end{abstract}

\maketitle
\setcounter{tocdepth}{1}
\tableofcontents

\section{Introduction}\label{sec_intro}

The \emph{tropical Grassmannian}, first studied in \cite{HKT, KT, Speyer_2004}, 
is the space
of \emph{realizable tropical linear spaces}, obtained by applying the valuation map to Puisseux-series valued
elements of the usual Grassmannian.
Meanwhile the \emph{Dressian} is the space of \emph{tropical Pl\"ucker vectors} 
$P = \{P_I\}_{I \in {[n] \choose k}}$, first studied by Andreas Dress,
who called them \emph{valuated matroids}.  Thinking of each tropical Pl\"ucker vector $P$ as a \emph{height
function} on the vertices of the hypersimplex $\Delta_{k,n}$, one can show that the Dressian 
parameterizes 
regular matroid subdivisions $\mathcal{D}_P$ 
of the hypersimplex \cite{Kapranov, Speyer}, which in turn are dual to the 
\emph{abstract tropical linear spaces} of the first author \cite{Speyer}.

There are positive notions of both of the above spaces.  
The \emph{positive tropical Grassmannian}, introduced by the authors in \cite{troppos}, is the space
of \emph{realizable positive tropical linear spaces}, obtained by applying the valuation map to 
Puisseux-series valued elements of the \emph{totally positive Grassmannian} \cite{postnikov, lusztig}.
The \emph{positive Dressian} is the space of \emph{positive tropical Pl\"ucker vectors}, and it 
was recently shown to parameterize the 
regular positroidal subdivisions of the hypersimplex \cite{LPW, PosConfig}.\footnote{Although 
this result did not appear in the literature until recently, it was anticipated by 
various people including the first author, Nick Early \cite{Early}, Felipe Rinc\`{o}n, Jorge Olarte.}

In general, the Dressian $\Dr_{k,n}$ is much larger than the 
tropical Grassmannian $\Trop Gr_{k,n}$ -- for example, 
the dimension of the Dressian $\Dr_{3,n}$ grows quadratically is $n$,
while the dimension of the tropical Grassmannian $\Trop Gr_{3,n}$ is 
linear in $n$
\cite{Herrmann2008HowTD}.  
However, the situation for their positive parts is different.
The first main result of this paper is the following, see \cref{thm:1}.
\begin{theorem*}
The positive tropical Grassmannian $\Trop^+Gr_{k,n}$ equals
the positive Dressian $\Dr^+_{k,n}$.\footnote{Our result was announced in 
	\cite[Theorem 9.6]{LPW}, and subsequently appeared in the independent work \cite{PosConfig}.}
\end{theorem*}

We give several interesting applications of \cref{thm:1}.  
The first application is a new proof of the following 
1987 conjecture of 
da Silva, which was proved in 2017 by Ardila,
Rinc\`{o}n and the second author~\cite{ARW2}, 
using the combinatorics of positroid polytopes.  
\begin{theorem*} \cite{ARW2}
Every positively oriented matroid
is realizable.
\end{theorem*}
\noindent Reformulating this statement in the language of Postnikov's 2006 
preprint~\cite{postnikov}, da Silva's conjecture 
says that every positively oriented matroid
is a \emph{positroid}.  
We give a new proof of this statement, using 
 \cref{thm:1}, which we think of as a ``tropical version" of da Silva's conjecture.
Interestingly, although the definitions
of positively oriented matroid and positroid don't involve tropical geometry at all, 
there does not seem to be an easy way to remove the tropical geometry from our proof 
without making it significantly longer.

There are two natural fan structures on the Dressian: the \emph{Pl\"ucker fan},
and the \emph{secondary fan,} which were shown in 
 \cite{Olarte} to coincide.
Our second application of \cref{thm:1} is a description of the maximal cones in the 
positive Dressian, or equivalently, the finest regular positroidal subdivisions of the hypersimplex.
The following result appears as \cref{thm:octahedron}.
\begin{theorem*}
Let $P$ be a positive tropical Pl\"ucker vector, and consider the 
corresponding 
regular positroidal subdivision $\mathcal{D}_P$. 
The following statements are equivalent:
	\begin{enumerate}
		\item $\mathcal{D}_P$ is a finest subdivision.
		\item Every facet of $\mathcal{D}_P$ is the matroid polytope of a series-parallel matroid.
		\item Every octahedron in $\mathcal{D}_P$ is subdivided.
	\end{enumerate}
\end{theorem*}

It was shown by the first author in \cite{Ktheory} that 
if $P$ is a tropical Pl\"ucker vector corresponding to a realizable tropical linear space,
$\mathcal{D}_P$ has 
at most $\frac{(n-c-1)!}{(k-c)!(n-k-c)!(c-1)!}$ interior faces of 
dimension $n-c$, with equality if and only if all facets of 
$\mathcal{D}_P$ correspond to series-parallel matroids.  We refer to this result as the 
\emph{$f$-vector theorem}.
Combining this result with \cref{thm:octahedron} 
gives the following elegant result (see \cref{cor:equality}): 
\begin{corollary*}
Every finest positroidal subdivision of $\Delta_{k,n}$ achieves
equality in the $f$-vector theorem.
In particular, such a positroidal 
	subdivision has precisely ${n-2 \choose k-1}$ facets (top-dimensional polytopes).
\end{corollary*}

Most of our paper concerns the 
\emph{regular} positroidal subdivisions of $\Delta_{k,n}$, which 
are precisely those induced by positive tropical Pl\"ucker vectors.
However, it is also natural to consider the set of \emph{all} 
positroidal subdivisions of $\Delta_{k,n}$, whether or not they are regular.
In light of the various nice realizability results for positroids, one might hope that 
all positroidal subdivisions of $\Delta_{k,n}$ are regular.
However, this is not the case.
In \cref{sec:nonregular}, 
we construct a nonregular positroidal subdivision of $\Delta_{3,12}$, based off 
a standard example of a nonregular mixed subdivision of $9 \Delta_2$.  
We also make a connection to the theory of tropical hyperplane arrangements 
and tropical oriented matroids~\cite{ArdilaDevelin, Horn}.

It is interesting to note that the positive tropical Grassmannian and the positive Dressian 
have recently appeared in the study of scattering amplitudes in 
$\mathcal{N}=4$ SYM \cite{Drummond:2019cxm, Arkani-Hamed:2019plo, Henke:2019hve, Early:2019eun, 
LPW, PosConfig},
and in certain scalar theories  \cite{Cachazo:2019ngv, Borges:2019csl}.
In particular, the second author together with Lukowski and Parisi \cite{LPW} gave striking evidence that 
 the positive tropical 
Grassmannian $\Trop^+Gr_{k+1,n}$ controls the regular positroidal subdivisions of 
the \emph{amplituhedron} $\mathcal{A}_{n,k,2} \subset Gr_{k,k+2}$, which was 
introduced by Arkani-Hamed and Trnka \cite{arkani-hamed_trnka}
to study scattering amplitudes in $\mathcal{N}=4$ SYM.

The structure of this paper is as follows.  In \cref{sec:posGrass} we review the notion of the 
positive Grassmannian and its cell decomposition, as well as matroid and positroid polytopes.
In \cref{sec:tropical}, after introducing the notions of the (positive) tropical Grassmannian 
and (positive) Dressian, we show that the positive tropical Grassmannian equals the positive Dressian.
We review the connection between the positive tropical Grassmannian and positroidal subdivisions
in \cref{sec:DP}, then give a new proof in \cref{sec:realizable} that every positively oriented matroid
is realizable.  We give several characterizations
of finest positroidal subdivisions of the hypersimplex in \cref{sec:finest}, and show that such 
subdivisions achieve equality in the $f$-vector theorem.  Then in \cref{sec:nonregular}, we 
construct a nonregular positroidal subdivision of $\Delta_{3,12}$, and make a connection
to the theory of tropical hyperplane arrangements and tropical oriented matroids
\cite{ArdilaDevelin, Horn}.  We end our paper with an appendix (\cref{app}), which reviews
some of Postnikov's technology \cite{postnikov} for studying positroids.

\textsc{Acknowledgements:}
This material is based upon work supported by the National Science Foundation
under agreement No.\ DMS-1855135, No.\ DMS-1854225,  No.\ DMS-1854316 and No.\ DMS-1854512.  Any opinions,
findings and conclusions or recommendations expressed in this material
are those of the authors and do not necessarily reflect  the
views of the National Science Foundation.

\section{The positive Grassmannian and positroid polytopes}\label{sec:posGrass}

\begin{defn} The {\itshape (real) Grassmannian} $Gr_{k,n}$ (for $0\le k \le n$) is the space of all $k$-dimensional subspaces of $\R^n$.  An element of
$Gr_{k,n}$ can be viewed as a $k\times n$ matrix of rank $k$ modulo invertible row operations, whose rows give a basis for the $k$-dimensional subspace.
\end{defn}

Let $[n]$ denote $\{1,\dots,n\}$, and $\binom{[n]}{k}$ denote the set of all $k$-element subsets of $[n]$. Given $V\in Gr_{k,n}$ represented by a $k\times n$ matrix $A$, for $I\in \binom{[n]}{k}$ we let $p_I(V)$ be the $k\times k$ minor of $A$ using the columns $I$. The $p_I(V)$ do not depend on our choice of matrix $A$ (up to simultaneous rescaling by a nonzero constant), and are called the {\itshape Pl\"{u}cker coordinates} of $V$.

\subsection{The positive Grassmannian and its cells}
\begin{defn}[{\cite[Section~3]{postnikov}}]\label{def:positroid}
	We say that $V\in Gr_{k,n}$ is {\itshape totally nonnegative} (respectively,
	\emph{totally positive}) if $p_I(V)\ge 0$ (resp. $p_I(V) > 0$) for all $I\in\binom{[n]}{k}$.  
	The set of all totally nonnegative $V\in Gr_{k,n}$ is the {\it totally nonnegative Grassmannian} $Gr_{k,n}^{\geq 0}$ and the set of all totally positive $V$
	is the \emph{totally positive Grassmannian} $Gr_{k,n}^{>0}$.
	For $M\subseteq \binom{[n]}{k}$, let $S_{M}$ be
the set of $V\in Gr_{k,n}^{\geq 0}$ with the prescribed collection of Pl\"{u}cker coordinates strictly positive (i.e.\ $p_I(V)>0$ for all $I\in M$), and the remaining Pl\"{u}cker coordinates
equal to zero (i.e.\ $p_J(V)=0$ for all $J\in\binom{[n]}{k}\setminus M$). If $S_M\neq\emptyset$, we call $M$ a \emph{positroid} and $S_M$ its \emph{positroid cell}.
\end{defn}

Each positroid cell $S_{M}$ is indeed a topological cell \cite[Theorem 6.5]{postnikov}, and moreover, the positroid cells of $Gr_{k,n}^{\ge 0}$ glue together to form a CW complex \cite{PSW}.

As shown in \cite{postnikov}, the cells of $Gr_{k,n}^{\geq 0}$
are in bijection 
with various combinatorial objects, including 
\emph{decorated permutations} $\pi$ on $[n]$ with $k$ anti-excedances, 
\emph{\Le -diagrams} $D$ of type $(k,n)$, and equivalence classes of \emph{reduced plabic graphs} $G$ of type $(k,n)$.
In \cref{app} we review these objects and give bijections between them.  This gives a canonical way to label each positroid by a decorated permutation, a \Le-diagram, and an equivalence class of plabic graphs; we will correspondingly refer to positroid cells as $S_{\pi}$, 
$S_D$, etc.






\subsection{Matroid and positroid polytopes}\label{sec:polytopes}


In what follows, we set 
$e_I := \sum_{i \in I} e_i \in \R^n$, where $\{e_1, \dotsc, e_n\}$ is the standard basis of $\RR^n$.

\begin{definition}\label{def:mpolytope}
Given a matroid $M=([n],\B)$, the (basis) \emph{matroid polytope} $\Gamma_M$ of $M$ is the convex hull of the indicator vectors of the bases of~$M$:
\[
\Gamma_M := \convex\{e_B \mid B \in \B\} \subset \RR^n.
\]
\end{definition}



The dimension of a matroid polytope is determined by the number of 
connected components of the matroid.
Recall that a matroid which cannot be written as the direct sum
of two nonempty matroids is called \emph{connected}.

\begin{proposition} \cite{Oxley}.
\label{prop:equiv}
Let $M$ be a matroid on $E$.  For two elements $a, b \in E$, we set
$a \sim b$ whenever there are bases $B_1, B_2$ of $M$ such that
	$B_2 = (B_1 - \{a\}) \cup \{b\}$.  The relation $\sim$ is an equivalence
relation, and the equivalence classes are precisely the connected components of $M$.
\end{proposition}

\begin{proposition}\label{prop:dim} \cite{Coxeter}
For any matroid,
 the dimension of its matroid polytope is 
$\dim \Gamma_M = n-c$, where $c$ is the number of connected components of 
$M$.  
\end{proposition}


Recall that any full
rank $k\times n$ matrix $A$ gives rise to a matroid 
$M(A)=([n],\B)$, where $\B = \{I \in \binom{[n]}{k} \ \vert \ p_I(A) \neq 0\}$.
\emph{Positroids} are the
matroids $M(A)$ associated to $k\times n$ matrices $A$ with 
maximal minors all nonnegative.  
We call the matroid polytope $\Gamma_M$ associated to a positroid a 
\emph{positroid polytope}.

\section{The positive tropical Grassmannian equals the positive Dressian}\label{sec:tropical}

In this section we 
review the notions of the tropical Grassmannian, the Dressian,
the positive tropical Grassmannian, and the positive Dressian.
The main theorem of this section is \cref{thm:1}, which says that 
the positive tropical Grassmannian equals the positive Dressian.


\begin{definition}\label{def:trophyper}
	Given $e=(e_1,\dots,e_N) \in \Z^N_{\geq 0}$, we let 
	$\mathbf{x}^e$ denote $x_1^{e_1} \dots x_N^{e_N}$.  
	Let  $E \subset \Z^N_{\geq 0}$.
	For $f = \sum_{e\in E} f_e \x^e$ a nonzero polynomial, 
	we denote by 
	$\Trop(f) \subset \R^N$ the set of all points 
	$(X_1,\dots, X_N)$ such that, if we form the collection of 
	numbers $\sum_{i=1}^N e_i X_i$ for $e$ ranging over $E$, then
	the minimum of this collection is not unique.
	We say that $\Trop(f)$ is the \emph{tropical hypersurface associated to $f$}.
\end{definition}

In our examples, we always consider polynomials $f$ with real coefficients.
We also have a positive version of \cref{def:trophyper}.

\begin{definition}
	Let  $E=E^+ \sqcup E^- \subset \Z^N_{\geq 0}$, and let
	 $f$ be a nonzero polynomial with real
	coefficients which we write as 
	$f = \sum_{e\in E^+} f_e \x^e - \sum_{e\in E^-} f_e \x^e$,
	where all of the coefficients $f_e$ are nonnegative real numbers.
	We  denote by 
	$\Trop^+(f) \subset \R^N$ the set of all points 
	$(X_1,\dots, X_N)$ such that, if we form the collection of 
	numbers $\sum_{i=1}^N e_i X_i$ for $e$ ranging over $E$, then
	the minimum of this collection is not unique and furthermore
	is achieved for some $e\in E^+$ and some $e\in E^-$.
	We say that $\Trop^+(f)$ is the \emph{positive part of 
	$\Trop(f)$.}
\end{definition}

The Grassmannian $Gr_{k,n}$ is a projective variety which can be embedded
in projective space $\PP^{\binom{[n]}{k}-1}$, and is cut out by the 
\emph{Pl\"ucker ideal}, that is, the ideal of relations satisfied by 
the Pl\"ucker coordinates of a generic $k \times n$ matrix.
These relations include the three-term Pl\"ucker relations, defined below.

\begin{definition}\label{def:3}
Let $1<a<b<c<d\leq n$ 
and choose a subset $S \in \binom{[n]}{k-2}$ which is disjoint from $\{a,b,c,d\}$.  
Then $p_{Sac} p_{Sbd} = p_{Sab} p_{Scd}+p_{Sad} p_{Sbc}$ is 
a \emph{three-term Pl\"ucker relation} for the Grassmannian $Gr_{k,n}$.
Here $Sac$ denotes $S \cup \{a,c\}$, etc.
\end{definition}

\begin{definition}\label{rem:tropPlucker}
Given $S, a, b, c, d$ as in \cref{def:3}, 
we say that the \emph{tropical three-term
Pl\"ucker relation holds} if 
\begin{itemize}
	\item $P_{Sac}+P_{Sbd} = P_{Sab}+P_{Scd} \leq P_{Sad}+P_{Sbc}$ 
		or 
	\item $P_{Sac}+P_{Sbd} = 
P_{Sad}+P_{Sbc} \leq 
P_{Sab}+P_{Scd}$
or 
\item $P_{Sab}+P_{Scd} = P_{Sad}+P_{Sbc}
\leq	P_{Sac}+P_{Sbd}$.
	\end{itemize}
	And we say that the \emph{positive tropical three-term
Pl\"ucker relation holds} if either of the first two conditions above holds.
\end{definition}

\begin{definition}\label{def:tropGrass}
	The \emph{tropical Grassmannian} $\Trop Gr_{k,n} \subset \R^{\binom{[n]}{k}}$ is 
	the intersection of 
	the tropical hypersurfaces $\Trop(f)$, where $f$ ranges over all
	elements of the Pl\"ucker ideal.
	The \emph{Dressian} $\Dr_{k,n} \subset 
	 \R^{\binom{[n]}{k}}$ is the intersection of 
	the tropical hypersurfaces $\Trop(f)$, where $f$ ranges over all
	three-term Pl\"ucker relations.
\end{definition}

 The \emph{tropical Grassmannian} $\Trop Gr_{k,n}$,
first studied in \cite{tropgrass, HKT, KT},
 parameterizes tropicalizations of ordinary linear 
 spaces, defined over the field of generalized Puisseux series
 $\mathbb{K}$ in one variable $t$, with real 
exponents.  More formally, recall that there is a valuation 
$\val_{\K}: \K \setminus \{0\} \to \R$, given by $\val_{\K}(c(t)) = \alpha_0$ if 
$c(t) = \sum c_{\alpha_m} t^{\alpha_m}$, where the lowest order term is assumed to have
non-zero coefficient $c_{\alpha_0} \neq 0$.
Then $P$ lies in the tropical Grassmannian $\Trop Gr_{k,n}$ if and only if 
there is an element $A = A(t) \in Gr_{k,n}(\K)$
 whose Pl\"ucker coordinates have valuations given by $P=\{P_I\}$ 
(see \cite{Payne1, Payne2} for a proof).
 We will call elements of $\Trop Gr_{k,n}$ \emph{realizable tropical linear spaces}.  The tropical Grassmannian is a 
 proper subset of the \emph{Dressian}\footnote{also called the \emph{tropical pre-Grassmannian} in \cite{tropgrass}
 and named in \cite{Herrmann2008HowTD} for Andreas Dress' work on valuated matroids}, which 
 parameterizes what one might call \emph{abstract tropical linear spaces}.
Moreover, the Dressian has a natural fan structure, whose cones correspond to
 the regular matroidal subdivisions of the hypersimplex
\cite{Kapranov},
 \cite[Proposition 2.2]{Speyer}, see \cref{prop:K}.
Note that the Dressian $Dr_{k,n}$ 
is the subset of $\R^{[n]\choose k}$
	where the tropical  three-term Pl\"ucker relations hold. 

\begin{definition}
\label{def:postropGrass}
	The \emph{positive tropical Grassmannian} $\Trop^+Gr_{k,n} \subset \R^{\binom{[n]}{k}}$ is the intersection of 
	the positive tropical hypersurfaces $\Trop^+(f)$, 
	where $f$ ranges over all elements of the Pl\"ucker ideal.
	The \emph{positive Dressian} $\Dr^+_{k,n} \subset 
	 \R^{\binom{[n]}{k}}$ is the intersection of 
	the positive tropical hypersurfaces $\Trop^+(f)$, 
	where $f$ ranges over all
	three-term Pl\"ucker relations.
\end{definition}

The {positive
 tropical Grassmannian} was
introduced by the authors fifteen years ago in \cite{troppos},
and was shown to parameterize 
 tropicalizations of ordinary linear 
 spaces that lie in the totally positive Grassmannian 
 (defined over the field of Puiseux series).
The positive tropical Grassmannian lies inside the {positive Dressian},
 which controls the regular {positroidal} subdivisions
 of the hypersimplex \cite{LPW}, see 
\cref{prop:positroidal}.
Note that the  positive Dressian $Dr^+_{k,n}$
is the subset of $\R^{[n]\choose k}$
	where the positive tropical three-term Pl\"ucker relations hold. 

\begin{definition}
	We say that a point $\{P_I\}_{I\in \binom{[n]}{k}}\in \R^{\binom{[n]}{k}}$
	is a \emph{(finite) tropical Pl\"ucker vector} 
	if it lies in the Dressian $\Dr_{k,n}$, i.e. 
	for every three-term Pl\"ucker relation, it lies in 
	the associated tropical hypersurface.
	And we say that $\{P_I\}_{I\in \binom{[n]}{k}}$
	is a \emph{positive tropical Pl\"ucker vector},  
	if it lies in the positive Dressian $\Dr^+_{k,n}$, i.e. 
	 for every three-term Pl\"ucker relation, it lies in 
	the positive part of the associated tropical hypersurface.
\end{definition}

\begin{example}
For $Gr_{2,4}$, there is only one Pl\"ucker relation, 
$p_{13} p_{24} = p_{12} p_{34}+p_{14} p_{23}$.  
	We have that $\Trop Gr_{2,4} = \Dr_{2,4} \subset \R^{\binom{[4]}{2}}$ is 
 the set of points
$(P_{12}, P_{13}, P_{14}, P_{23}, P_{24}, P_{34})\in \R^6$ such that 
\begin{itemize}
	\item $P_{13}+P_{24} = P_{12}+P_{34} \leq P_{14}+P_{23}$ 
		or 
	\item $P_{13}+P_{24} = 
	P_{14}+P_{23} \leq 
	P_{12}+P_{34}$
	or 
\item $P_{12}+P_{34} = P_{14}+P_{23}
\leq	P_{13}+P_{24}$.
	\end{itemize}
	And $\Dr^+_{2,4} = \Trop^+Gr_{2,4} \subset \R^{\binom{[4]}{2}}$ is  
	 the set of points
	$(P_{12}, P_{13}, P_{14}, P_{23}, P_{24}, P_{34})\in \R^6$ such that 
	\begin{itemize}
		\item $P_{13}+P_{24} = P_{12}+P_{34} \leq P_{14}+P_{23}$ 
			or 
		\item $P_{13}+P_{24} = 
	P_{14}+P_{23} \leq 
	P_{12}+P_{34}$
	\end{itemize}
\end{example}

In general, the Dressian $\Dr_{k,n}$ is much larger than the 
tropical Grassmannian $\Trop Gr_{k,n}$ -- for example,
the dimension of the Dressian $\Dr_{3,n}$ grows quadratically is $n$,
while the dimension of the tropical Grassmannian $\Trop Gr_{3,n}$ is 
linear in $n$ \cite{Herrmann2008HowTD}.  
However, the situation for their positive parts is different.
The main result of this section is the following.

\begin{theorem}\label{thm:1}
The positive tropical Grassmannian $\Trop^+Gr_{k,n}$ equals
the positive Dressian $\Dr^+_{k,n}$.
\end{theorem}

\cref{thm:1} was recently announced in \cite{LPW}.  It subsequently 
appeared in independent work of \cite{PosConfig}.

Before proving \cref{thm:1}, we review some results from \cite{troppos} 
which allow one to compute positive tropical varieties.

\begin{remark}\label{tropicalparam}
In \cref{app} we describe many parametrizations of cells of $(Gr_{k,n})_{\geq 0}$,
which were given by Postnikov using plabic graphs.
 \cite[Proposition 2.5]{troppos} says that 
if one has a subtraction-free rational map $f$ which surjects onto 
the positive part $V^+(J)$ of a variety (for example a cluster chart), 
then the tropicalization of this map 
surjects onto the positive tropical part $\Trop^+V(J)$ of the variety.
Therefore 
we can tropicalize each parameterization $\Phi_G$ from \cref{network_param} -- 
 to obtain a parameterization of 
a positive tropical positroid variety (in particular, 
$\Trop^+Gr_{k,n}$).
More specifically, we tropicalize $\Phi_G$ by 
replacing the positive parameters $x_{\mu}$ (with $\prod_\mu x_{\mu} = 1$) with 
real parameters $X_{\mu}$ (with $\sum_{\mu} X_{\mu} = 0$) -- and replacing products with sums and sums
with minimums in the expressions for flow polynomials.  Then 
\cite[Proposition 2.5]{troppos} say that this tropicalized map $\Trop \Phi_G$
gives a parameterization of $\Trop^+ Gr_{k,n}$. 
\end{remark}

For the proof of \cref{thm:1} it is convenient to use one particular plabic graph
(corresponding to the directed graph $\Web_{k,n}$ from \cite[Section 3]{troppos}),
see \cref{web36}.
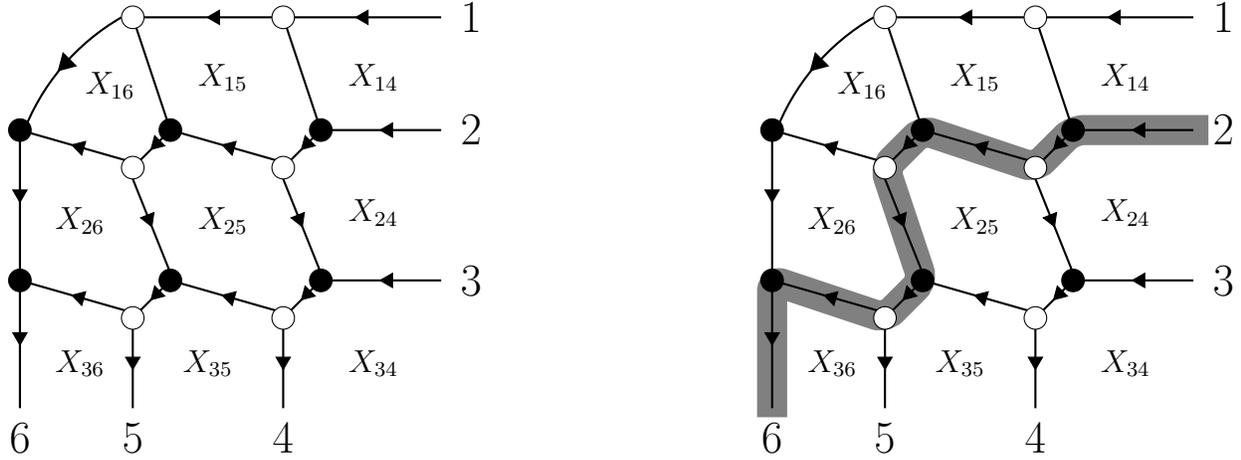
\begin{figure}[h]
\begin{tikzpicture}
\node at (2.2,6.6) {$X_{16}$};
\node at (3.7,6.7) {$X_{15}$};
\node at (5.7,6.7) {$X_{14}$};
\node at (1.8,4.8) {$X_{26}$};
\node at (3.7,4.8) {$X_{25}$};
\node at (5.7,4.9) {$X_{24}$};
\node at (1.8,2.9) {$X_{36}$};
\node at (3.5,2.9) {$X_{35}$};
\node at (5.7,2.9) {$X_{34}$};
\node at (7,7.5) {\Large 1};
\node at (7,6) {\Large 2};
\node at (7,4) {\Large 3};
\node at (4.5,1.9) {\Large 4};
\node at (2.5,1.9) {\Large 5};
\node at (1,1.9) {\Large 6};

\begin{scope}[thick, every node/.style={sloped,allow upside down}]
  \draw (2.35,7.5) arc (120:160:3);
  \draw[-triangle 60] (1.52,6.8) -- (1.5,6.78);
  \draw (4.35,7.5)-- node {\midarrow} (2.65,7.5);
  \draw (6.6,7.5)-- node {\midarrow} (4.35,7.5);
  \draw (6.6,6)-- node {\midarrow} (5,6);
  \draw (6.6,4)-- node {\midarrow} (5,4);
  \draw (1,6)-- node {\midarrow} (1,4.1);
  \draw (2.5,5.35)-- node {\midarrow} (3,4.1);
  \draw (4.5,5.35)-- node {\midarrow} (5,4.1);
  \draw (1,4)-- node {\midarrow} (1,2.3);
  \draw (2.5,3.35)-- node {\midarrow} (2.5,2.3);
  \draw (4.5,3.35)-- node {\midarrow} (4.5,2.3);
  \draw (4.9,5.9)-- node {\midarrow} (4.6,5.6);
  \draw (2.9,5.9)-- node {\midarrow} (2.6,5.6);
  \draw (4.9,3.9)-- node {\midarrow} (4.6,3.6);
  \draw (2.9,3.9)-- node {\midarrow} (2.6,3.6);
  \draw (4.4,5.6)-- node {\midarrow} (3,6);
  \draw (2.4,5.6)-- node {\midarrow} (1,6);
  \draw (4.4,3.6)-- node {\midarrow} (3,4);
  \draw (2.4,3.6)-- node {\midarrow} (1,4);
  \draw (2.5,7.5)-- (3,6);
  \draw (4.5,7.5)-- (5,6);
\end{scope}

\filldraw[fill=white] (2.5,7.5) circle (0.15cm);
\filldraw[fill=white] (4.5,7.5) circle (0.15cm);
\filldraw[fill=black] (1,6) circle (0.15cm);
\filldraw[fill=black] (3,6) circle (0.15cm);
\filldraw[fill=black] (5,6) circle (0.15cm);
\filldraw[fill=white] (2.5,5.5) circle (0.15cm);
\filldraw[fill=white] (4.5,5.5) circle (0.15cm);
\filldraw[fill=black] (1,4) circle (0.15cm);
\filldraw[fill=black] (3,4) circle (0.15cm);
\filldraw[fill=black] (5,4) circle (0.15cm);
\filldraw[fill=white] (2.5,3.5) circle (0.15cm);
\filldraw[fill=white] (4.5,3.5) circle (0.15cm);

\draw[rounded corners, gray, line width=4mm] 
  (16.8,6) -- (15,6) -- (14.5,5.5) -- (13,6) -- (12.5,5.5)
  -- (13,4) -- (12.5,3.5) -- (11,4) -- (11,2.18);
  
\node at (12.2,6.6) {$X_{16}$};
\node at (13.7,6.7) {$X_{15}$};
\node at (15.7,6.7) {$X_{14}$};
\node at (11.8,4.8) {$X_{26}$};
\node at (13.7,4.8) {$X_{25}$};
\node at (15.7,4.9) {$X_{24}$};
\node at (11.8,2.9) {$X_{36}$};
\node at (13.5,2.9) {$X_{35}$};
\node at (15.7,2.9) {$X_{34}$};
\node at (17,7.5) {\Large 1};
\node at (17,6) {\Large 2};
\node at (17,4) {\Large 3};
\node at (14.5,1.9) {\Large 4};
\node at (12.5,1.9) {\Large 5};
\node at (11,1.9) {\Large 6};

\begin{scope}[thick, every node/.style={sloped,allow upside down}]
  \draw (12.35,7.5) arc (120:160:3);
  \draw[-triangle 60] (11.52,6.8)-- (11.5,6.78);
  \draw (14.35,7.5)-- node {\midarrow} (12.65,7.5);
  \draw (16.6,7.5)-- node {\midarrow} (14.35,7.5);
  \draw (16.6,6)-- node {\midarrow} (15,6);
  \draw (16.6,4)-- node {\midarrow} (15,4);
  \draw (11,6)-- node {\midarrow} (11,4.1);
  \draw (12.5,5.35)-- node {\midarrow} (13,4.1);
  \draw (14.5,5.35)-- node {\midarrow} (15,4.1);
  \draw (11,4)-- node {\midarrow} (11,2.3);
  \draw (12.5,3.35)-- node {\midarrow} (12.5,2.3);
  \draw (14.5,3.35)-- node {\midarrow} (14.5,2.3);
  \draw (14.9,5.9)-- node {\midarrow} (14.6,5.6);
  \draw (12.9,5.9)-- node {\midarrow} (12.6,5.6);
  \draw (14.9,3.9)-- node {\midarrow} (14.6,3.6);
  \draw (12.9,3.9)-- node {\midarrow} (12.6,3.6);
  \draw (14.4,5.6)-- node {\midarrow} (13,6);
  \draw (12.4,5.6)-- node {\midarrow} (11,6);
  \draw (14.4,3.6)-- node {\midarrow} (13,4);
  \draw (12.4,3.6)-- node {\midarrow} (11,4);
  \draw (12.5,7.5)-- (13,6);
  \draw (14.5,7.5)-- (15,6);
\end{scope}

\filldraw[fill=white] (12.5,7.5) circle (0.15cm);
\filldraw[fill=white] (14.5,7.5) circle (0.15cm);
\filldraw[fill=black] (11,6) circle (0.15cm);
\filldraw[fill=black] (13,6) circle (0.15cm);
\filldraw[fill=black] (15,6) circle (0.15cm);
\filldraw[fill=white] (12.5,5.5) circle (0.15cm);
\filldraw[fill=white] (14.5,5.5) circle (0.15cm);
\filldraw[fill=black] (11,4) circle (0.15cm);
\filldraw[fill=black] (13,4) circle (0.15cm);
\filldraw[fill=black] (15,4) circle (0.15cm);
\filldraw[fill=white] (12.5,3.5) circle (0.15cm);
\filldraw[fill=white] (14.5,3.5) circle (0.15cm);

\end{tikzpicture}
\caption{$\Web_{k,n}$ for $k=3$ and $n=6$.  If $w$ is the path on the right-hand side, then
	$\wt(w) = x_{25}x_{24}x_{36} x_{35}x_{34}$ and 
	$\Wt(w) = X_{25}+X_{24}+X_{36} +X_{35}+X_{34}.$}
   \label{web36}
 \end{figure}

Applying \cref{network_param} to the graph 
from \cref{web36}, we have the following result.
\begin{theorem}
	Label the faces of $G_0:= Web_{k,n}$ by indices $\mu$ and let $\mathcal{P}_{k,n}$ denote
	the collection of indices.
	Define the weight $\wt(w)$ of a path $w$ in $Web_{k,n}$ to 
	be the product of parameters $x_{\mu}$ where $\mu$ ranges over all
	face labels to the left of a path.  Define the weight of a \emph{flow} (i.e. a collection of 
	nonintersecting paths) to be the product of the weights of its paths.
	Let $p_J^{G_0} = \sum_F \wt(F)$ where $F$ ranges over all flows from $\{1,2,\dots,k\}$ to 
	$J$.  Then the map $\Phi:=\Phi_{G_0}$ sending 
	$(x_{\mu})_{\mu\in \mathcal{P}_{k,n}} 
	\in (\R_{>0})^{k(n-k)}$ to the collection of \emph{flow polynomials}
	$\{p_J^{G_0}\}_{J\in {[n]\choose k}}$ is a 
	 homemorphism from 
	$(\R_{>0})^{k(n-k)}$ to 
	the totally positive Grassmannian $(Gr_{k,n})_{>0}$ (realized in its Pl\"ucker 
	embedding).
\end{theorem}


In the case of the graph $G_0 = \Web_{k,n}$, we obtain the following parameterization of 
$\Trop^+ Gr_{k,n}$. 

\begin{theorem}
	Label the faces of $Web_{k,n}$ by indices $\mu$ as before.  
	Define the weight $\Wt(w)$ of a path $w$ in $Web_{k,n}$ to 
	be the sum of parameters $X_{\mu}$ where $\mu$ ranges over all
	face labels to the left of a path.  Define the weight of a \emph{flow} (i.e. a collection of 
	nonintersecting paths) to be sum of the weights of its paths.
	Let $P_J^{G_0} = \min_F \Wt(F)$ where $F$ ranges over all flows from $\{1,2,\dots,k\}$ to 
	$J$.  Then the map $\Trop \Phi_G= \Trop \Phi_{G_0}$ sending 
	$(X_{\mu})_{\mu\in \mathcal{P}_{k,n}} 
	\in (\R)^{k(n-k)}$ to the collection of \emph{tropical flow polynomials}
	$\{P_J^{G_0}\}_{J\in {[n]\choose k}}$  
	is a bijection from 
	$\R^{k(n-k)}$ to 
	the tropical positive Grassmannian $\Trop^+ Gr_{k,n}$ (realized in its Pl\"ucker 
	embedding).
\end{theorem}



In the case of $G_0 = \Web_{k,n}$, we can easily invert the maps $\Phi:=\Phi_{G_0}$ and $\Trop \Phi = \Trop \Phi_{G_0}$.
This was done in \cite{troppos}; we review the construction here.
First, given $i$ and $j$ labeling horizontal and vertical
wires of $\Web_{k,n}$ (i.e. $1 \leq i \leq k$ and $k+1 \leq j \leq n$),
let 
\[K(i,j):=\{1,2,\dots,i-1\} \cup\{i+j-k,i+j-k+1,\dots,j-1,j\}. \]
If $(i,j)$ does not correspond to a region of $\Web_{k,n}$,
set $K(i,j):=[k]$.

\begin{definition}\label{def:invert}
	Let $p = \{p_K\}_{K\in {[n]\choose k}}\in \R_{>0}^{[n]\choose k}$.
	Then for $i$ and $j$ labeling horizontal and vertical
	wires of $\Web_{k,n}$ (i.e. $1 \leq i \leq k$ and 
	$k+1 \leq j \leq n$), we define
	$$\Psi(p)_{(i,j)}:=\frac{p_{K(i,j)} p_{K(i+1,j-2)}  p_{K(i+2,j-1)}}{p_{K(i,j-1)}  p_{K(i+1,j)}p_{K(i+2,j-2)})}.$$

We likewise define the tropical version.
	Let $P = \{P_K\}_{K\in {[n]\choose k}}\in \R^{[n]\choose k}$.
\[ \Trop \Psi(P)_{(i,j)} = \left( P_{K(i,j)} + P_{K(i+1,j-2)} +  P_{K(i+2,j-1)} \right) - \left( P_{K(i,j-1)} + P_{K(i+1,j)} + P_{K(i+2,j-2)}) \right). \]
\end{definition}

\cref{def:invert} gives a way to label each face of $\Web_{k,n}$
by a (tropical) Laurent monomial in (tropical) Pl\"ucker coordinates.  
  This is shown in \cref{invweb36}.

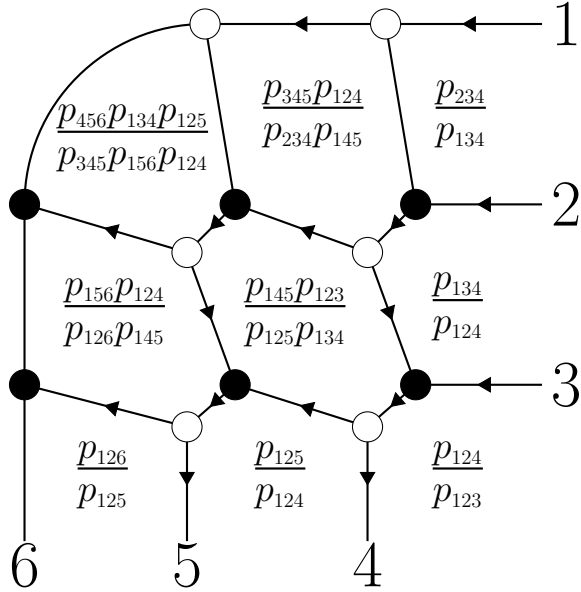
\begin{figure}[h]
	\begin{tikzpicture}[scale = 0.8]
\node at (2.8,5.1) {\LARGE $\frac{p_{\scaleto{456}{5pt}}p_{\scaleto{134}{5pt}}p_{\scaleto{125}{5pt}}}{p_{\scaleto{345}{5pt}}p_{\scaleto{156}{5pt}}p_{\scaleto{124}{5pt}}}$};
\node at (5.8,5.5) {\LARGE $\frac{p_{\scaleto{345}{5pt}}p_{\scaleto{124}{5pt}}}{p_{\scaleto{234}{5pt}}p_{\scaleto{145}{5pt}}}$};
\node at (2.5,2.2) {\LARGE $\frac{p_{\scaleto{156}{5pt}}p_{\scaleto{124}{5pt}}}{p_{\scaleto{126}{5pt}}p_{\scaleto{145}{5pt}}}$};
\node at (5.5,2.2) {\LARGE $\frac{p_{\scaleto{145}{5pt}}p_{\scaleto{123}{5pt}}}{p_{\scaleto{125}{5pt}}p_{\scaleto{134}{5pt}}}$};
\node at (8.25,5.5) {\LARGE $\frac{p_{\scaleto{234}{5pt}}}{p_{\scaleto{134}{5pt}}}$};
\node at (8.2,2.3) {\LARGE $\frac{p_{\scaleto{134}{5pt}}}{p_{\scaleto{124}{5pt}}}$};
\node at (2.3,-0.5) {\LARGE $\frac{p_{\scaleto{126}{5pt}}}{p_{\scaleto{125}{5pt}}}$};
\node at (5.25,-0.5) {\LARGE $\frac{p_{\scaleto{125}{5pt}}}{p_{\scaleto{124}{5pt}}}$};
\node at (8.2,-0.5) {\LARGE $\frac{p_{\scaleto{124}{5pt}}}{p_{\scaleto{123}{5pt}}}$};
\node at (10,7) {\huge 1};
\node at (10,4) {\huge 2};
\node at (10,1) {\huge 3};
\node at (6.7,-2) {\huge 4};
\node at (3.7,-2) {\huge 5};
\node at (1,-2) {\huge 6};

\begin{scope}[thick, every node/.style={sloped,allow upside down}]
  \draw (4,7) arc (90:180:3);
  \draw (7,7)-- node {\midarrow} (4,7);
  \draw (9.6,7)-- node {\midarrow} (7,7);
  \draw (9.6,4)-- node {\midarrow} (7.5,4);
  \draw (9.6,1)-- node {\midarrow} (7.5,1);
  \draw (4,7)-- (4.5,4);
  \draw (7,7)-- (7.5,4);
  \draw (1,4)-- (1,1);
  \draw (1,1)-- (1,-1.6);
  \draw (3.7,0.3)-- node {\midarrow} (3.7,-1.6);
  \draw (6.7,0.3)-- node {\midarrow} (6.7,-1.6);
  \draw (3.7,3.2)-- node {\midarrow} (4.5,1);
  \draw (6.7,3.2)-- node {\midarrow} (7.5,1);
  \draw (4.5,4)-- node {\midarrow} (3.7,3.2);
  \draw (7.5,4)-- node {\midarrow} (6.7,3.2);
  \draw (4.5,1)-- node {\midarrow} (3.7,0.3);
  \draw (7.5,1)-- node {\midarrow} (6.7,0.3);
  \draw (3.7,3.2)-- node {\midarrow} (1,4);
  \draw (6.7,3.2)-- node {\midarrow} (4.5,4);
  \draw (3.7,0.3)-- node {\midarrow} (1,1);
  \draw (6.7,0.3)-- node {\midarrow} (4.5,1);
\end{scope}

\filldraw[fill=white] (4,7) circle (0.25cm);
\filldraw[fill=white] (7,7) circle (0.25cm);
\filldraw[fill=black] (1,4) circle (0.25cm);
\filldraw[fill=black] (4.5,4) circle (0.25cm);
\filldraw[fill=black] (7.5,4) circle (0.25cm);
\filldraw[fill=white] (3.7,3.2) circle (0.25cm);
\filldraw[fill=white] (6.7,3.2) circle (0.25cm);
\filldraw[fill=black] (1,1) circle (0.25cm);
\filldraw[fill=black] (4.5,1) circle (0.25cm);
\filldraw[fill=black] (7.5,1) circle (0.25cm);
\filldraw[fill=white] (3.7,0.3) circle (0.25cm);
\filldraw[fill=white] (6.7,0.3) circle (0.25cm);

\end{tikzpicture}
\caption{Inverting the map}
   \label{invweb36}
 \end{figure}

\begin{proposition}
The maps $\Phi : \RR_{>0}^{k(n-k)} \to Gr^+_{k,n}$ and $\Psi : Gr^+_{k,n} \to \RR_{>0}^{k(n-k)}$ are inverses.
\end{proposition}

\begin{proposition}
\cite[Corollary 3.5 and its proof]{troppos}
The maps $\Trop \Phi : \RR^{k(n-k)} \to \Trop^+ Gr_{k,n}$ and $\Trop \Psi : \Trop^+ Gr_{k,n} \to \RR^{k(n-k)}$ are inverses.
\end{proposition}

\begin{lemma}\label{rem:corectangles}
The collection of Pl\"ucker coordinates 
	$\mathcal{C} = \{p_{K(i,j)} \ \vert \ 1 \leq i \leq k, k+1 \leq j \leq n\}$ 
form a \emph{cluster} for the 
\emph{cluster algebra structure} \cite{Scott}
on (the affine cone over the)
Grassmannian $Gr_{k,n}$.  
We call this the \emph{corectangles 
cluster}.  
In particular, 
this collection of Pl\"ucker coordinates is algebraically independent, and 
 all other Pl\"ucker coordinates 
can be written as Laurent polynomials with positive coefficients
in the Pl\"ucker coordinates from the collection.
\end{lemma}
\begin{proof}
Note that for each $i$ and $j$ as above,
$K(i,j)$ is a $k$-element subset of $[n]$.  
Moreover, if we identify Young diagrams contained in a $k \times (n-k)$
rectangle with the labels of the vertical steps in the 
length-$n$ lattice path taking unit steps south and west
from $(k,n-k)$ to $(0,0)$, then the elements $K(i,j)$ precisely
correspond to the Young diagrams $\lambda$ whose complementary
Young diagram is a rectangle.  It is not hard to see that 
the collection $\{K(i,j)\}$ is a maximal \emph{weakly separated
set collection} \cite{OPS}, and hence 
form a \emph{cluster} for the 
\emph{cluster algebra structure} \cite{Scott}.
\end{proof}
\begin{example}\label{ex:invert}
\cref{invweb36} depicts the map $\Trop \Psi$.  
Since $\Trop \Phi$ and 
	$\Trop \Psi$ 
	are inverses, 
this example shows how to express each of the variables 
	$X_{ij}$ (as shown in 
	\cref{web36}) in terms of the tropical Pl\"ucker coordinates
$\mathcal{C} = \{P_{K(i,j)} \ \vert \ 1 \leq i \leq k, k+1 \leq j \leq n\}$. 
 Note moreover that 
if we choose a normalization in tropical projective space
(e.g. where $P_{123}=0$), then we can solve for the
	tropical Pl\"ucker coordinates in $\mathcal{C}$ 
	in terms of the $X_{ij}$'s.
For example, comparing 
\cref{web36} and \cref{invweb36},
we see that if $P_{123}=0$, then 
$P_{124}-P_{123} = P_{124} = X_{34}$, 
$P_{134}-P_{124} = X_{24}$, so 
$P_{134} = X_{34}+X_{24}$, etc.
In this example we see that from the collection 
$\{X_{ij}\}$ together with the normalization 
$P_{123}=0$, we can uniquely determine the 
Pl\"ucker coordinates 
$\{124, 125, 134, 145\} \cup 
\{123, 234, 345, 456, 156, 126\}$.
As in \cref{rem:corectangles},
this collection of Pl\"ucker coordinates is a \emph{cluster}
for the cluster algebras structure on the Grassmannian.
\end{example}

It is easy to generalize \cref{ex:invert}, obtaining the following result.
\begin{lemma}\label{lem:injective}
The map 
$\Trop \Psi$ sending 
$\mathcal{C} = \{P_{K(i,j)} \ \vert \ 1 \leq i \leq k, k+1 \leq j \leq n\}$
	(with the convention that $P_{12\dots k} = 0$)
to $\{\Trop \Psi(P)_{(i,j)}\}$ is an injective map from 
$\R^{k(n-k)}$ to $\R^{k(n-k)}$.
\end{lemma}
%
%
%
%
%
%
%
\begin{proof}[Proof of \cref{thm:1}]
To prove \cref{thm:1}, we must show that every point in the positive Dressian also lies 
in the positive tropical Grassmannian.
We will consider any tropical Pl\"ucker vector $P=\{P_K\}_{K\in {[n]\choose k}} \in \Dr^+_{k,n}$, with the normalization $P_{12\dots k} = 0$,
and compute $Q:= (\Trop \Phi) \circ (\Trop \Psi)(P)$. 
	This will give an (a priori new) realizable tropical Pl\"ucker vector in $\Trop^+ Gr_{k,n}$. 
We must show that $Q= P$. 

Recall that the map $\Trop \Psi$ depends only on the tropical Pl\"ucker coordinates in 
$\mathcal{C} = \{P_{K(i,j)} \ \vert \ 1 \leq i \leq k, k+1 \leq j \leq n\}$, mapping them 
	to $\{\Trop \Psi(P)_{(i,j)}\}$.  Moreover from \cref{lem:injective}, 
	$\Trop \Psi$ is an injective map from $\R^{k(n-k)}$ to $\R^{k(n-k)}$.  Therefore,
	since $\Trop \Psi$ and $\Trop \Phi$ are inverses, 
	we have that  $Q_{K(i,j)} = P_{K(i,j)}$ for all $K(i,j)$ with $1 \leq i \leq k$ and $k+1 \leq j \leq n$.
But now 
from \cref{rem:corectangles}, the collection $\{P_{K(i,j)}\}$
is  a cluster for the cluster structure on 
	$Gr_{k,n}$.  And by 
	\cite{OhSpeyer}, all Pl\"ucker clusters can be obtained
	from each via three-term Pl\"ucker relations.  Since every 
	Pl\"ucker coordinate lies in a Pl\"ucker cluster \cite{OPS}, 
	 the three-term Pl\"ucker relations alone (which we 
	 know are satisfied since $P\in \Dr^+_{k,n}$) 
	 determine all the other values 
$P_K$ and $Q_K$ for $K \in {[n] \choose k}$, so we must have $P_K = Q_K$ for all 
	$K \in {[n]\choose k}$.
Therefore $P=K$ and we are done. 
\end{proof}

\begin{remark}
One may generalize \cref{thm:1} and its proof to any positroid cell, using 
the $\Le$-network associated to a positroid cell and the inverse map from 
	\cite{Talaska2}.
\end{remark}

\section{The positive tropical Grassmannian and positroidal subdivisions}\label{sec:DP}

Recall that $\Delta_{k,n}$ denotes the $(k,n)$-hypersimplex, defined as the 
convex hull of the points $e_I$ where $I$ runs over $\binom{[n]}{k}$.
Consider a real-valued function $\{I\} \mapsto P_I$ on the vertices
of $\Delta_{k,n}$.  We define a polyhedral subdivision $\mathcal{D}_P$
of $\Delta_{k,n}$ as follows: consider the points
$(e_I, P_I)\in \Delta_{k,n} \times \R$ and take their convex hull. 
Take the lower faces (those whose outwards normal vector have last component
negative) and project them back down to $\Delta_{k,n}$; this gives us 
the subdivision $\mathcal{D}_P$.  We will omit the subscript $P$ when it is clear from context.  A subdivision obtained in this manner is called \emph{regular}.  

\begin{remark}\label{rem:lower}
A lower face $F$ of the regular subdivision defined above is determined by some 
	vector $\lambda = (\lambda_1,\dots,\lambda_n,-1)$ whose dot product with 
	the vertices of the $F$ is maximized.
	So if $F$ is the matroid polytope of a matroid $M$ with bases $\mathcal{B}$, 
	this is equivalent to saying that 
	$\lambda_{i_1} + \dots + \lambda_{i_k}-P_I = \lambda_{j_1} + \dots + \lambda_{j_k}-P_J >
	\lambda_{h_1} + \dots + \lambda_{h_k} - P_H$ for any two bases $I,J \in \mathcal{B}$
	and $H\notin \mathcal{B}$.
\end{remark}

Given a subpolytope $\Gamma$ of $\Delta_{k,n}$, we say that $\Gamma$ is \emph{matroidal}
if the vertices of $\Gamma$, considered as elements of $\binom{[n]}{k}$, are the 
bases of a matroid $M$, i.e. $\Gamma = \Gamma_M$.

The following result is originally due to Kapranov \cite{Kapranov}; it was 
also 
proved in \cite[Proposition 2.2]{Speyer}.
\begin{theorem}\label{prop:K}
	The following are equivalent.
	\begin{itemize}
		\item The collection $\{P_I\}_{I \in \binom{[n]}{k}}$ 
			is a tropical Pl\"ucker vector.
		\item The one-skeleta of $\mathcal{D}_P$ and $\Delta_{k,n}$
			are the same.
		\item Every face of $\mathcal{D}_P$ is matroidal.
	\end{itemize}
\end{theorem}

Given a subpolytope $\Gamma$ of $\Delta_{k,n}$, we say that $\Gamma$ is \emph{positroidal}
if the vertices of $\Gamma$, considered as elements of $\binom{[n]}{k}$, are the 
bases of a positroid $M$, i.e. $\Gamma = \Gamma_M$.
The positroidal version of \cref{prop:K} was recently proved in 
\cite{LPW}, and independently in \cite{PosConfig}.

\begin{theorem}\label{prop:positroidal}
The following are equivalent.
\begin{itemize}
	\item The collection $\{P_I\}_{I \in \binom{[n]}{k}}$ 
		is a positive tropical Pl\"ucker vector.
	\item Every face of $\mathcal{D}_P$ is positroidal.
\end{itemize}
\end{theorem}

It follows from \cref{prop:positroidal}
that the regular subdivisions of $\Delta_{k+1,n}$ consisting of positroid polytopes
are precisely those of the form $\mathcal{D}_P$, where $P=\{P_I\}$ is a positive tropical
Pl\"ucker vector. 

%
%
%
%
%
%
%

\section{A new proof that positively oriented matroids are realizable}\label{sec:realizable}

In 1987, da Silva \cite{daS} conjectured that every positively oriented matroid
is realizable.  Reformulating this statement in the language of Postnikov's 2006 
preprint~\cite{postnikov}, her conjecture says that every positively oriented matroid
is a positroid.  
In 2017, da Silva's conjecture was proved by Ardila, Rinc\`{o}n and the second author~\cite{ARW2}, 
using the combinatorics of positroid polytopes.  In this section we will 
give a new proof of the conjecture, using 
our \cref{thm:1}, which we think of as a ``tropical version" of da Silva's conjecture.

Recall that an oriented matroid of rank $k$ on $[n]$ can be specified by its 
\emph{chirotope}, which is a function from $[n]^k$ 
to $\{ -, 0, + \}$ obeying certain axioms~\cite{OrientedMatroidBook}.
If $M$ is a full rank $k \times n$ real matrix, the function taking $(i_1, i_2, \ldots, i_k)$ to the sign of the minor using columns $(i_1, i_2, \dots, i_k)$ is a chirotope, and the realizable oriented matroids are precisely the chirotopes occurring in this way.
Thus, if $M$ represents a point of the totally nonnegative Grassmannian, then $M$ gives a chirotope $\chi$ with $\chi(i_1, i_2, \ldots, i_k) \in \{ 0, + \}$ for $1 \leq i_1 < i_2 < \cdots < i_k \leq n$.

We define a \emph{positively oriented matroid} to be a chirotope $\chi : [n]^k \to \{ -,0, + \}$ 
such that  $\chi(i_1, i_2, \ldots, i_k) \in \{ 0, + \}$ for $1 \leq i_1 < i_2 < \cdots < i_k \leq n$.
Since every positroid gives rise to a positively oriented matroid, to prove da Silva's conjecture,
we need to verify that every positively oriented matroid comes from a positroid, or in other 
words, is realizable.

\begin{theorem}\cite[Theorem 5.1]{ARW2}\label{thm:realizable}
Let $M$ be a positively oriented matroid of rank $k$ on the ground set $[n]$. Then $M$ is realizable.
\end{theorem}

Before proving \cref{thm:realizable}, we need the following lemma, which was implicit in 
\cite[Section 4]{Speyer}.
\begin{lemma}\label{lem:realizable}
	Suppose that $P = \{P_I\}$ lies in the tropical Grassmannian $\Trop Gr_{k,n}$.  Then every face of the 
	matroidal subdivision $\mathcal{D}_P$ of $\Delta_{k,n}$ corresponds to a realizable matroid.
\end{lemma}
\begin{proof}
Let $\Gamma_M$ be a face of $\mathcal{D}_P$, and let $\mathcal{B}$ denote the 
bases of the matroid $M$.
Adding an affine linear function to $P$, we may assume that $P_I$ is $0$ for $I \in \mathcal{B}$; convexity then implies that $P_I > 0$ for $I \not\in \mathcal{B}$.

Since $P$ lies in the tropical Grassmannian, we can choose a $\K$-valued $k \times n$ matrix 
	$A = A(t)$ whose Pl\"ucker coordinates have valuations given by $P=\{P_I\}$ (see the discussion 
	following 
	\cref{def:tropGrass}).
	But now if we 
	set $t=0$, then the matrix $A(0)$ has Pl\"ucker coordinates which are nonzero for $I\in \mathcal{B}$
	and zero for $I\notin \mathcal{B}$.  Therefore  $M$ is a realizable matroid.
\end{proof}

\begin{proof}[Proof of \cref{thm:realizable}]
We first use the matroid $M$ to construct a point of the Dressian, 
following the method of~\cite[Proposition 4.4]{Speyer}
Namely, let $\rho_M$ be the rank function of the matroid $M$,
and for $I \in \binom{[n]}{k}$, set $P_I = - \rho_M(I)$. 
Then \cite[Proposition 4.4]{Speyer} implies that 
	 $P:=\{P_I\}_{I \in {[n] \choose k}}$ is a point of the Dressian,
	 and that the matroid polytope $\Gamma_M$ is a face of the subdivision
	 $\mathcal{D}_P$.

	 Using the fact that $M$ is positively oriented, we will show that 
	 $P$ is in fact a point of the positive Dressian.
Indeed, consider any $(k-2)$-element subset $S$ of $[n]$ and any $a<b<c<d$ in $[n] \setminus S$. 
We need to show that 
	\begin{equation*}\label{eq:1}
		P_{Sac}+P_{Sbd} = \min(P_{Sab}+P_{Scd}, P_{Sad}+P_{Sbc}), 
	\end{equation*}
	or equivalently, that 
	\begin{equation*}\label{eq:2}
		\rho_M(Sac)+\rho_M(Sbd) = \max(\rho_M(Sab)+\rho_M(Scd), \rho_M(Sad)+\rho_M(Sbc)). 
	\end{equation*}
	
	Let $M'$ be the matroid $(M/S)|_{\{ a,b,c,d \}}$ on the ground set 
	$\{ a,b,c,d \}$.  For $x$, $y\in \{a,b,c,d \}$, we have $\rho_M(Sxy) = \rho_M(S) + \rho_{M'}(xy)$. 
	Thus, we need to show that 
	\begin{equation}\label{eq:3}
		\rho_{M'}(ac) + \rho_{M'}(bd) = \max(\rho_{M'}(ab) + \rho_{M'}(cd), \rho_{M'}(bc)+\rho_{M'}(ad)). 
	\end{equation}

Now we claim that $M'$, being a minor of a positively oriented matroid,
is itself a positively oriented matroid.  It is easy to verify that the dual of a 
positively oriented matroid is again a positively oriented matroid, and moreover, 
\cite[Lemma 4.11]{ARW2}  showed
        that positively oriented matroids are closed under restriction.  
An analogous proof shows that positively oriented matroids are closed under contraction.
This verifies the claim.

It now remains to verify \eqref{eq:3} for 
all positively oriented matroids on four elements, which is routine.

We now know that $P=\{P_I\}$ lies in the positive Dressian, 
so \cref{thm:1} shows that $P$ is in the positive tropical Grassmannian.
	But now by \cref{lem:realizable}, this implies that 
every face of the matroidal subdivision $\mathcal{D}_P$ of $\Delta_{k,n}$ 
	corresponds to a realizable matroid. 
In particular, we have $P_I \leq 0$ with equality if and only if $I$ is a basis of $M$, so $\Gamma_M$ is a face of $\mathcal{D}_P$, and we have shown that $\Gamma_M$ is realizable.
\end{proof}

Interestingly, although the definitions
of ``positively oriented matroid" and ``positroid" don't involve tropical geometry at all, there does not seem to be a way to remove the tropical geometry from our proof 
without making it significantly longer.

\section{Finest positroidal subdivisions of the hypersimplex}\label{sec:finest}

In this section we show that finest positroidal subdivisions of 
the hypersimplex $\Delta_{k,n}$
achieve equality in the first author's \emph{$f$-vector theorem.}

\begin{definition}
A matroid is called \emph{series-parallel} if it can be obtained
by repeated series-parallel extensions from the matroid corresponding
to a generic point of $Gr_{1,2}$.
\end{definition}
See
\cite[Section 6.4]{White} for background on series-parallel matroids.

\begin{theorem}\label{thm:equality} \cite{Ktheory}
Let $P$ be a tropical Pl\"ucker vector arising as $\val(p_I(A))$ for 
some $A\in Gr_{k,n}(\K)$.  Then $\mathcal{D}_P$ has 
at most $\frac{(n-c-1)!}{(k-c)!(n-k-c)!(c-1)!}$ interior faces of 
dimension $n-c$, with equality if and only if all facets of 
$\mathcal{D}_P$ correspond to series-parallel matroids.
\end{theorem}

In particular, the number of \emph{facets} of $\mathcal{D}_P$ -- 
that is, the number of matroid polytopes of dimension $n-1$ in $\mathcal{D}_P$ -- 
is at most ${n-2 \choose k-1}$.


The following result can be found in \cite[Corollary 11.2.15]{Oxley}.
\begin{theorem}
A connected matroid is series-parallel if and only if it has no
minor which is the uniform matroid $U_{2,4}$ or the graphical
matroid $M_{K_4}$ associated to the complete graph $K_4$.
\end{theorem}

The graphical matroid $M_{K_4}$ is not a positroid, and all minors
of positroids are positroids \cite{ARW}, so we have the following
corollary.

\begin{corollary}\label{cor:seriesparallel}
A connected positroid is series-parallel if and only if it has
	no uniform matroid $U_{2,4}$ as a minor.
\end{corollary}

If $M$ is a matroid on the ground set $[n]$, with matroid
polytope $\Gamma_M$, and $I$ and $J$ are disjoint subsets of $[n]$, then the the polytope $\Gamma_{M \backslash I / J}$ is $\Gamma_M \cap \{ z_i =0 : i \in I \} \cap \{ z_j = 1 : j \in J \}$. So we can phrase Corollary~\ref{cor:seriesparallel} as
\begin{corollary}\label{cor:seriesparallelgeometric}
Let $M$ be a connected positroid.  Then $M$ is series-parallel
if and only if its matroid polytope
        $\Gamma_M$ does not contain any face which is an
        (unsubdivided) octahedron.
\end{corollary}

It follows from \cref{prop:dim} 
that in a matroidal subdivision,
all facets correspond to connected matroids.  

\begin{theorem}\label{thm:octahedron}
Let $P = \{P_K\}_{K\in {[n] \choose k}}$ be a positive tropical Pl\"ucker vector.
For the positroidal subdivision $\mathcal{D}_P$ of $\Delta_{k,n}$, the 
following are equivalent:
	\begin{enumerate}
		\item $\mathcal{D}_P$ is a finest subdivision.
		\item Every facet of $\mathcal{D}_P$ is the matroid polytope of a series-parallel matroid.
		\item Every octahedron in $\mathcal{D}_P$ is subdivided.
	\end{enumerate}
\end{theorem}

\begin{proof}
Suppose that (3) holds.  Let $\Gamma_M$ be a facet of this subdivision. Since $\dim \Gamma_M = n-1$, 
	the matroid $M$ is connected, and by hypothesis $M$ is a positroid. Hypothesis (3) says that $\Gamma_M$ 
	does not contain any octahedron, so Corollary~\ref{cor:seriesparallelgeometric} says that $M$ is series-parallel. We have shown $(2)$.

Now suppose that (2) holds.  If every facet is series-parallel,
then by \cref{thm:equality}, we get equality in the $f$-vector theorem,
and in particular get equality in the $c=1$ term.  So we have the 
maximal number of possible facets, so the positroidal subdivision is finest 
	possible.  This implies (1).

	Now suppose that (1) holds. To show that every octahedron in $\mathcal{D}_P$ is subdivided,
we need to show that we never have equality in a tropical 3-term Pl\"ucker relation, in other
words, we never have 
$$P_{Sab}+P_{Scd} = P_{Sad}+P_{Sbc}$$ for 
$a<b<c<d$ and $S\in {[n] \choose k-2}$ disjoint from $\{a,b,c,d\}$.

Using the fact that the positive Dressian equals the positive tropical Grassmannian 
	(\cref{thm:1}),
as well as \cref{tropicalparam},
	we can use flows in plabic graphs to parameterize the points in the positive Dressian, as in 
\cref{network_param}.  We note that it follows from the technology 
of \cite{PSW} 
	that a flow is uniquely determined by its weight $\wt(F)$
	(compare Definition 4.3 and Table 1, and note that flows
	are in bijection with almost perfect matchings). 

Let us choose a reduced plabic graph $G$ for $(Gr_{k,n})_{>0}$, i.e. a reduced plabic graph
	with trip permutation $(k+1,k+2,\dots,n, 1, 2, \dots,k)$, and choose a perfect orientation
	$\mathcal{O}$ with sources at $I_{\mathcal{O}} = Sab$.  (The fact that we can do so follows from e.g. 
	\cref{prop:perf}).

Then by 
\cref{tropicalparam}, we can express $P = \Trop \Phi_G(\{X_{\mu}\}),$
for some fixed real values $X_{\mu}$ labeling the faces of $G$.
In particular, the coordinates of $P = \{P_K\}_{K\in {[n] \choose k}}$ can be expressed as 
$P_K = \min_F (\Wt(F)),$ where $F$ ranges over all flows from $Sab$ to $K$, 
and $\Wt(F)$ is a sum of certain parameters $X_{\mu}$.

Since we are assuming that $\mathcal{D}_P$ is finest, we can assume that the 
parameters $X_{\mu}$ are generic: that these parameters are distinct real numbers,
and that there are not two different subsets of parameters whose sums coincide.

Let us consider the tropical Pl\"ucker coordinate $P_{Sab}$.  
This equals $\min_F (\Wt(F)),$ where $F$ ranges over all flows from $Sab$ to $Sab$; 
in this case, the flows $F$ are simply collections of vertex-disjoint cycles in $G$ (including 
the empty collection). We now explain how to reduce to the case that the flow achieving the minimum is the empty flow.

Let $F'$ be the flow achieving the minimum, so $F'$ is a collection of disjoint cycles. 
Adjust $\mathcal{O}$ to a new perfect orientation $\mathcal{O}'$ by reversing the orientation of all edges belonging to $F'$.
Then $\mathcal{O'}$ is again a perfect orientation
(see \cite[Lemma 4.5]{PSW}) 
and that (preserving the values of the $X_{\mu}$) 
the collection of new Pl\"ucker coordinates are all adjusted by the same scalar (the weight of $F'$),
preserving the point in tropical projective space which is represented by $P$.
Now, in the orientation $\mathcal{O}'$, the minimum flow for $P_{Sab}$ is the empty flow.
We therefore assume, from now on, that the minimum  flow for $P_{Sab}$ is the empty flow. With this reduction, we have $P_{Sab}=0$.

Meanwhile $P_{Scd}$ is the weight of the  minimal flow $F_2$ from 
$Sab$ to $Scd$, which will be
	a pair of paths $\{w_1,w_2\}$ taking $a$ to $d$ and $b$ to $c$ (plus possibly some closed loops).
$P_{Sad}$ is the weight of the minimal flow $F_3$ from $Sab$ to $Sad$, 
which will be a single path $w_3$ from $b$ to $d$ (plus possibly closed loops).
And $P_{Sbc}$ is the weight of the minimal flow $F_4$ from $Sab$ to $Sbc$,
which will be a single path $w_4$ from $a$ to $c$ (plus possibly closed loops), see \cref{Flows}.
	
\begin{figure}[h]
	\begin{tikzpicture}[scale=0.89]

\filldraw[fill=white, thick] (0,5) circle (1.5cm);
\node at (0,2) {\LARGE $P_{Sab}$};
\filldraw[fill=black] (-0.2,6.5) circle (0.08cm);
\filldraw[fill=black] (0.9,6.2) circle (0.08cm);
\filldraw[fill=black] (1.01,3.9) circle (0.08cm);
\filldraw[fill=black] (-0.5,3.6) circle (0.08cm);
\node at (-0.2,6.8) {\large $a$};
\node at (1.13,6.4) {\large $b$};
\node at (1.21,3.8) {\large $c$};
\node at (-0.5,3.3) {\large $d$};

\node at (2.5,5) {\huge +};

\filldraw[fill=white, thick] (5,5) circle (1.5cm);
\node at (5,2) {\LARGE $P_{Scd}$};
\filldraw[fill=black] (4.8,6.5) circle (0.08cm);
\filldraw[fill=black] (5.9,6.2) circle (0.08cm);
\filldraw[fill=black] (6.01,3.9) circle (0.08cm);
\filldraw[fill=black] (4.5,3.6) circle (0.08cm);
\node at (4.8,6.8) {\large $a$};
\node at (6.13,6.4) {\large $b$};
\node at (6.21,3.8) {\large $c$};
\node at (4.5,3.3) {\large $d$};
\draw[thick] plot [smooth] coordinates {
    (4.8,6.5) (4.9,6.2) (4.6,5.7) (4.8,5.2) (4.4,4.6) (4.5,3.8)};
\draw[-triangle 60] (4.45,4.1) -- (4.5,3.76);
\draw[thick] plot [smooth] coordinates {
    (5.9,6.2) (5.7,5.7) (5.9,5.3) (5.8,4.9) (6.1,4.5) (6.01,4.1)};
\draw[-triangle 60] (6.05,4.25) -- (6,4.05);
\node at (4.2,5) {\large $w_1$};
\node at (5.5,5.4) {\large $w_2$};
\draw[thick] plot [smooth cycle] coordinates {
    (5.1,5) (5.25,4.95) (5.3,4.7) (5.2,4.4) (5.3,4.1) (5.2,3.9) (5.1,3.9) (4.9,4.4)};
\draw[-triangle 60] (4.9,4.4) -- (4.95,4.65);

\node at (7.5,5) {\huge =};
\node at (7.5,5.5) {\LARGE ?};

\filldraw[fill=white, thick] (10,5) circle (1.5cm);
\node at (10,2) {\LARGE $P_{Sad}$};
\filldraw[fill=black] (9.8,6.5) circle (0.08cm);
\filldraw[fill=black] (10.9,6.2) circle (0.08cm);
\filldraw[fill=black] (11.01,3.9) circle (0.08cm);
\filldraw[fill=black] (9.5,3.6) circle (0.08cm);
\node at (9.8,6.8) {\large $a$};
\node at (11.13,6.4) {\large $b$};
\node at (11.21,3.8) {\large $c$};
\node at (9.5,3.3) {\large $d$};
\draw[thick] plot [smooth] coordinates {
    (10.9,6.2) (10.6,5.9) (10.7,5.1) (10.4,4.8) (9.8,4.4) (9.55,3.75)};
\draw[-triangle 60] (9.56,3.79) -- (9.54,3.73);
\node at (10.5,4.4) {\large $w_3$};
\draw[thick] plot [smooth cycle] coordinates {
    (9.2,5.67) (9.2,5.07) (9.5,4.97) (9.8,5.27) (10.2,5.37) (10.25,5.77) (9.7,5.92)};
\draw[-triangle 60] (9.8,5.92) -- (9.76,5.93);

\node at (12.5,5) {\huge +};

\filldraw[fill=white, thick] (15,5) circle (1.5cm);
\node at (15,2) {\LARGE $P_{Sbc}$};
\filldraw[fill=black] (14.8,6.5) circle (0.08cm);
\filldraw[fill=black] (15.9,6.2) circle (0.08cm);
\filldraw[fill=black] (16.01,3.9) circle (0.08cm);
\filldraw[fill=black] (14.5,3.6) circle (0.08cm);
\node at (14.8,6.8) {\large $a$};
\node at (16.13,6.4) {\large $b$};
\node at (16.21,3.8) {\large $c$};
\node at (14.5,3.3) {\large $d$};
\draw[thick] plot [smooth] coordinates {
    (14.8,6.5) (14.8,5.8) (15.8,4.8) (16,4.07)};
\draw[-triangle 60] (16,4.1) -- (16.01,4.05);
\node at (15.3,5.85) {\large $w_4$};

\end{tikzpicture}
\caption{
Flows $F_1, F_2, F_3, F_4$ used to compute $P_{Sab}, P_{Scd}, P_{Sad}, P_{Sbc}$.}
\label{Flows}
\end{figure}
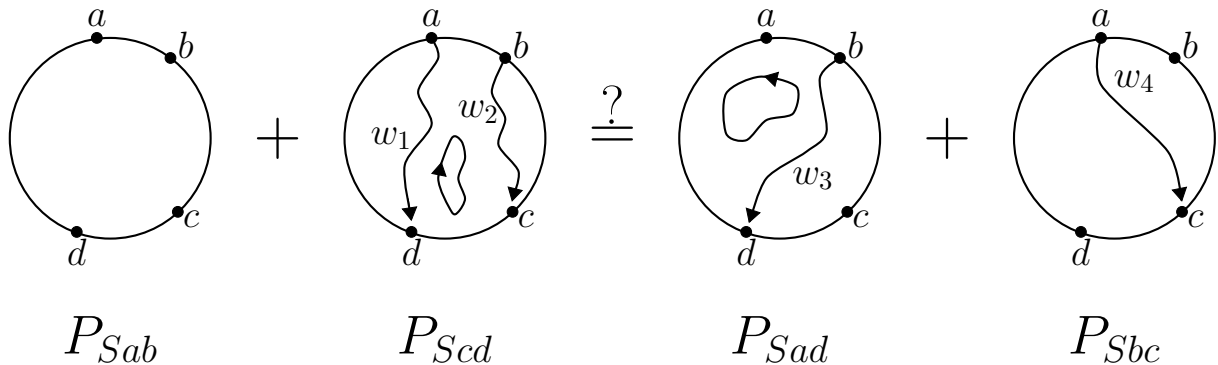

But now because our parameters $X_{\mu}$ associated to the faces are generic, the only way 
to get equality $P_{Sab}+P_{Scd} = P_{Sad}+P_{Sbc}$ is if our minimal flow $F_2$ from 
$S_{ab}$ to $S_{cd}$ has to its left precisely the same multiset of faces that the pair of 
flows $(F_3,F_4)$ (which consists of the paths $w_3,w_4$ plus possibly some loops) does.
This is only possible if $w_1$ and $w_2$ are obtained from $w_3$ and $w_4$ by ``switching tails''
at an intersection point of $w_3$ and $w_4$.  But then $\{w_1, w_2\}$ would not be vertex-disjoint
and hence not part of a flow.
\end{proof}

Combining \cref{thm:1}, 
\cref{thm:equality}, and 
\cref{thm:octahedron}, we now have the following.
\begin{corollary}\label{cor:equality}
Every finest positroidal subdivision of $\Delta_{k,n}$ achieves
equality in the $f$-vector theorem.  
	In particular, such a positroidal 
subdivision has precisely ${n-2 \choose k-1}$ facets.
\end{corollary}

\section{Nonregular positroidal subdivisions}\label{sec:nonregular}

In this paper we have discussed
the positive Dressian, which consists of weight functions on the vertices of the 
hypersimplex $\Delta_{k,n}$ which induce positroidal subdivisions of $\Delta_{k,n}$;
recall that subdivisions induced by weight functions are 
called \emph{regular} or \emph{coherent}. 
It is also natural to consider the set of \emph{all} 
positroidal subdivisions of $\Delta_{k,n}$, 
whether or not they are regular.
(See \cite{Triangulations} for background on regular subdivisions.)
In this section, we will construct a nonregular positroidal subdivision of $\Delta_{3,12}$, 
and also make a connection to the theory of tropical hyperplane arrangements and tropical oriented matroids~\cite{ArdilaDevelin, Horn}.

Our strategy for producing the counterexample is as follows.  
We will start with a standard example of a nonregular rhombic tiling of 
a hexagon (with side lengths equal to $3$), and extend it to a
nonregular mixed subdivision of 
$9 \Delta_2$; this mixed subdivision gives rise to a dual arrangement
$\mathcal{H}$ of $9$ \emph{tropical pseudohyperplanes} in $\TT \PP^2$.  
Moreover,
the mixed subdivision corresponds, via the \emph{Cayley trick}, to a polyhedral
subdivision of $\Delta_2 \times \Delta_8$.  We then 
map this polyhedral subdivision
to a matroidal subdivision of $\Delta_{3,12}$, and analyze the $0$-dimensional
regions of $\mathcal{H}$ to show that it is a positroidal subdivision of
$\Delta_{3,12}$.
Note that \cite[Example 4.7]{Herrmann2008HowTD} used a similar strategy to encode
a nonregular matroidal subdivision of $\Delta_{3,9}$.  We give a careful
exposition here in order to verify that our subdivision is positroidal.

\subsection{The product of simplices and the hypersimplex}
Let $I$ be any $k$-element subset of $[n]$ and let $J = [n] \setminus I$.
Let $\Poly_I \subset \Delta_{k,n}$ be the convex hull of all points
of the form $e_I - e_i + e_j$ for $i \in I$ and $j \in J$; clearly this set of 
points is in bijection with $I \times J$.
The polytope $\Poly_I$ is isomorphic to $\Delta_{k-1} \times \Delta_{n-k-1}$, 
with vertices in bijection with $I \times J$.  $\Poly_I$ has dimension $n-2$
and sits inside $\Delta_{k,n}$, which has dimension $n-1$.
We review standard constructions for passing between polyhedral subdivisions of $\Poly_I$ and matroidal subdivisions of $\Delta_{k,n}$. 
We will be interested in polyhedral subdivisions of $\Poly_I$ all of whose vertices are vertices of $\Poly_I$, and we will take the phrase ``subdivision of $\Poly_I$" to include this condition. 

In many references, $I$ is standardized to be $[k]$. However, we will want to keep track of how these standard constructions relate to the property of a matroid being a positroid and, for this purpose, it will be important how $I$ sits inside the circularly ordered set $[n]$, so we do not impose a standard choice of $I$.

Given a matroidal subdivision $\cD$ of $\Delta_{k,n}$, we can intersect $\cD$ with $\Poly_I$ and obtain a polyhedral subdivision $\cG_I$ of $\Poly_I$. 
If $\cD$ is regular, so is $\cG_I$.

\subsection{From 
 subdivisions of $\Poly_I$ to subdivisions of $\Delta_{k,n}$}
Following \cite[Theorem 7 and Remark 8]{Dressian}, as well as \cite{Felipe},
 we will explain how to 
map each convex hull of vertices of $\Poly_I$ to 
a matroid polytope 
 inside
$\Delta_{k,n}$; this will be the matroid polytope of a \emph{principal 
transversal matroid}. 

Let $X \subseteq I \times J$. 
We define a polytope 
$\gamma(X) = \mathrm{Hull}_{(i,j) \in X} (e_I - e_i + e_j) \subseteq \Poly_I$.
We also define a bipartite graph $G(X)$ with vertex set $I \sqcup J = [n]$ and an edge from $i \in I$ to $j \in J$ if and only if $(i,j) \in X$.

Associated to the graph $G(X)$ is the \emph{principal transversal matroid} $\Trans(G(X))$ (see \cite{Brualdi} and \cite[Chapter 7]{White}),
 defined as follows: 
 $B$ is a basis of $\Trans(G(X))$ if and only if there is a matching of $I \setminus B$ to $J \cap B$ in the bipartite graph $G(X)$. The matroid 
 $\Trans(G(X))$ is realized by  a
 $k \times n$ matrix $A = A_X$, with rows labeled by $I$ and columns labeled by $[n]$ where:
 \begin{itemize}
	 \item the values $A_{i j}$ for $(i, j) \in X$ (where $i\in I$ and $j\in J$)
		 are algebraically independent, 
	 \item $A_{i j} = 0$ if $(i, j) \not\in X$ (where $i\in I$ and $j \in J$),
	 \item $A_{i i'}=\delta_{i i'}$ (where $i, i'\in I$).
 \end{itemize}
\begin{remark}\label{rem:identity}
Note that the restriction of $A$ to the columns labeled by $I$ is 
the $k \times k$ identity matrix.
\end{remark}
In terms of polyhedral geometry, the matroid polytope of $\Trans(G(X))$ is the intersection of $\Delta_{k,n}$ with $e_I + \Span_{\RR_{\geq 0}} \{ e_j - e_i : (i,j) \in X \}$. 
Summarizing, we have the following.

\begin{lemma}\label{lem:trans}
Each polytope 
$\gamma(X) = \mathrm{Hull}_{(i,j) \in X} (e_I - e_i + e_j) \subseteq \Poly_I$
gives rise to 
 the matroid polytope $\Gamma_{\Trans(G(X))} 
\subseteq \Delta_{k,n}$.  Abusing notation, we say that 
	$\Trans$ maps $\gamma(X)$ to $\Gamma_{\Trans(G(X))}$.
\end{lemma}

If $\cG$ is a polyhedral subdivision of $\Poly_I$, then we can apply $\Trans$ to each polytope in $\cG$.

\begin{proposition}\cite[Theorem 7 and Remark 8]{Dressian} and \cite{Felipe}
If we apply $\Trans$ to each polytope in a 
polyhedral subdivision $\cG$ of $\Poly_I$, then 
we will obtain a matroid subdivision $\Trans(\cG)$ of $\Delta_{k,n}$.
 The subdivision $\Trans(\cG)$ is regular if and only if $\cG$ is. 
\end{proposition}

We will eventually be studying triangulations of $\Poly_I$, so we will want to focus on the case that $\gamma(X)$ is an $(n-2)$-dimensional simplex.

\begin{lemma} \label{SimplexTreeSP}
Let $X \subseteq I \times J$. The following are equivalent:
\begin{enumerate}
\item The polytope $\gamma(X)$ is an $(n-2)$-dimensional simplex.
\item The graph $G(X)$ is a tree on the vertices $[n]$.
\end{enumerate}
\end{lemma}

\begin{proof} 
	The equivalence of (1) and (2) is simple.  The polytope
	$\gamma(X)$ is an $(n-2)$-dimensional simplex if and only if 
	it's the convex hull of $(n-1)$ affinely independent points.
	But this is equivalent to the statement that $G(X)$ consists of 
	$n-1$ edges and no subset forms a cycle.  This means that 
	$G(X)$ is a tree on $[n]$.
\end{proof}

\begin{remark}
The conditions from \cref{SimplexTreeSP} are additionally equivalent
to the condition that 
the matroid $\Trans(G(X))$ is series-parallel.  One can prove this 
	using e.g. \cite[Proposition 5.1]{Speyer}. 
\end{remark}

	We will want to know when the matroids in Lemma~\ref{SimplexTreeSP} are positroidal.  One direction of \cref{PlanarTree} comes from 
	\cite[Theorem 6.3]{Marcott}.
\begin{lemma} \label{PlanarTree}
Suppose that $X$ is a subset of  $I \times J$ such that $G(X)$ is a tree. The matroid $\Trans(G(X))$ is positroidal if and only if we can embed the tree $G(X)$ in a disk so that it is planar, and its vertices lie on the boundary of the disk in the standard circular order on $I \sqcup J = [n]$.
\end{lemma}

\begin{proof}
If $G(X)$ can be embedded as a planar tree in a disk as above,
	then this graph is \emph{noncrossing}, and by \cite[Theorem 6.3]{Marcott}
the transversal matroid $\Trans(G(X))$ is a positroid.

On the other hand, if it cannot be embedded as a planar tree, then 
we can find $i_1, i_2 \in I$ and $j_1, j_2 \in J$, such that 
$(i_1,j_1)$ and $(i_2,j_2)$ lie in $X$, and when we put the numbers
	$\{i_1,i_2,j_1,j_2\}$ at the boundary of a disk in the standard
	circular order, the two chords
	$(i_1,j_1)$ and $(i_2,j_2)$ cross each other.  Moreover since 
$G(X)$  is a tree, we cannot have both 
$(i_1,j_2)$ and $(i_2,j_1)$ in $X$.  
Without loss of generality we can assume that either $i_1<i_2 < j_1<j_2$
or $i_1 < j_2 < j_1 < i_2$.  Let us consider the first case.
	Then if we look at the rows labeled by $\{i_1,i_2\}$ and the 
	columns labeled by 
	$\{i_1,i_2,j_1,j_2\}$ in the matrix $A = A_X$, we find that 
	the minors $p_{i_1 i_2 }$ and $p_{j_1 j_2}$ are nonzero, but the product
	$p_{i_1 j_1} p_{i_2 j_2}$ is zero.  This fails to be a positroid on 
	$\{i_1, i_2, j_1, j_2\}$ because such conditions are incompatible
	with finding a non-negative solution to the 
	Pl\"ucker relation $p_{i_1 j_1} p_{i_2 j_2} = 
	p_{i_1 i_2} p_{j_1 j_2}+p_{i_1 j_2}p_{i_2 j_1}$.
	Using \cref{rem:identity}, we can now extend 
	this $2 \times 4$ submatrix of $A$ to a $k \times (k+2)$ submatrix
	of $A$, by adding the rows and columns indexed by 
	$I \setminus \{i_1,i_2\}$
	The second case is analogous.
\end{proof}

\subsection{From tropical pseudohyperplane arrangements to 
subdivisions of the product of simplices}

We now explain how to go between tropical pseudohyperplane arrangements and subdivisions of $\Poly_I$.
This section is based on \cite{ArdilaDevelin}, which initiated the study
of tropical oriented matroids and conjectured that they are in 
bijection with subdivisions of the product of two simplices. 
\cite{ArdilaDevelin} proved their conjecture in the case of 
$\Delta_{k-1} \times \Delta_2$, which is all we need here;  \cite{Horn}
 proved their conjecture in general.  Consult these sources for more detail.

\begin{figure}
	\begin{tikzpicture}[scale=0.8]

\draw[->, ultra thick] (0,0) -- (4.5,0);
\draw[->, ultra thick] (0,0) -- (-3,-2);
\draw[->, ultra thick] (0,0) -- (0,4.5);

\node at (-2,1.5) {\huge 1};
\node at (1,-2) {\huge 2};
\node at (2,2) {\huge 3};

\node at (0.4,0.3) {\large 123};
\node at (-1,-1.1) {\large 12};
\node at (0.3,2.3) {\large 13};
\node at (2.3,-0.3) {\large 23};

\end{tikzpicture}
\caption{The labeling of the regions of a tropical hyperplane} \label{TropHyperplaneRegions}
\end{figure}
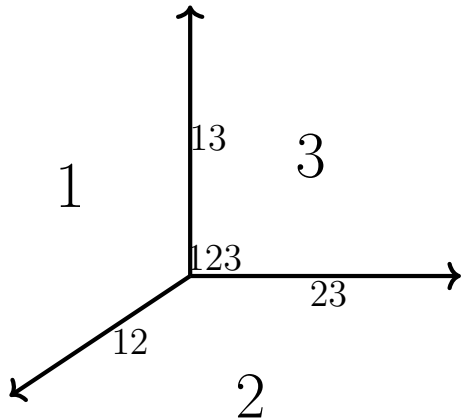
Let $\TT\PP^{k-1}$ denote \emph{tropical projective space} $\RR^k/\RR(1,1,\ldots,1)$, and let $c = (c_1, \ldots, c_k)$ be 
an element of $\TT \PP^{k-1}$. The \emph{tropical hyperplane} $H_c$ 
centered at $c$ is the set of points $(x_1, x_2, \ldots, x_k) \in \TT \PP^{k-1}$ such that 
$\min_{1 \leq j \leq k} \{x_j - c_j\}$ is not unique. 
If 
$x=(x_1, x_2, \ldots, x_k)$ is any point of $\TT \PP^{k-1}$, we 
let $S(H,x)$ be the set of indices $j \in [k]$ at which $x_j - c_j$ is minimized. 
\cref{TropHyperplaneRegions} shows a tropical hyperplane in $\TT \PP^2$,
where the horizontal and vertical coordinates are $x_1-x_3$ and $x_2-x_3$, and 
each region is labelled with the set $S(H,x)$ for $x$ in that region.
An \emph{arrangement of $m$ labelled tropical hyperplanes} is a list of $m$ tropical hyperplanes in $\TT \PP^{k-1}$. 

A \emph{tropical pseudohyperplane} $H$ is a subset of $\TT \PP^{k-1}$ which is 
PL-homeomorphic to a tropical hyperplane.  
Note that the quantity $S(H,x)$ makes sense for $H$ a tropical pseudohyperplane in $\TT \PP^{k-1}$ and $x \in \TT \PP^{k-1}$. 
An \emph{arrangement of $m$ labelled tropical
pseudohyperplanes} is a list of $m$ tropical pseudohyperplanes
 which intersect in ``reasonable" ways, 
see \cite[Section 5]{Horn} for details. 
Our main focus in this section will be  on the case of 
 tropical 
pseudohyperplanes in $\TT \PP^2$. 

Consider an arrangement of $n-k$ tropical pseudohyperplanes $H_1$, $H_2$, \dots, $H_{n-k}$ in $\TT \PP^{k-1}$. 
Given a point $x \in \TT \PP^{k-1}$, we define a subset $X(x)$ of $[k] \times [n-k]$ where $(i,j) \in X(x)$ if and only if $j \in S(H_i, x)$. 
We can thus associate to each $x\in \TT \PP^{k-1}$ 
a polytope $\gamma(X(x)) \subseteq \Delta_{k-1} \times \Delta_{n-k-1}$, 
as well as the matroid polytope $\Gamma_{\Trans(G(X(x)))}$ of the 
transversal matroid $\Trans(G(X(x)))$.
 If we let $x$ range over the bounded regions of the 
 tropical pseudohyperplane arrangement, we obtain the interior regions of a 
  subdivision of $\Delta_{k-1} \times \Delta_{n-k-1}$. Using
   \cite[Theorem 1]{DevelinSturmfels} and \cite[Theorems 1.2 and 1.3]{Horn}, this subdivision 
  is regular if and only if 
  tropical pseudohyperplane arrangement can be realized by genuine tropical hyperplanes.

\subsection{Our counterexample} 

\begin{figure}
	\begin{tikzpicture}[scale=0.5]

\filldraw[fill=gray] (3,8) rectangle ++(1,1);

\draw[thick] (0,9) -- (9,9);
\draw[thick] (9,9) -- (9,0);
\draw[thick] (0,9) -- (9,0);

\draw[thick] (6,9) -- (9,6);
\draw[thick] (5,9) -- (6,8);
\draw[thick] (7,7) -- (9,5);
\draw[thick] (4,9) -- (5,8);
\draw[thick] (6,7) -- (8,5);
\draw[thick] (4,8) -- (6,6);
\draw[thick] (7,5) -- (8,4);
\draw[thick] (2,9) -- (3,8);
\draw[thick] (4,7) -- (6,5);
\draw[thick] (7,4) -- (9,2);
\draw[thick] (1,9) -- (5,5);
\draw[thick] (6,4) -- (9,1);

\draw[thick] (1,8) -- (4,8);
\draw[thick] (5,8) -- (9,8);
\draw[thick] (2,7) -- (4,7);
\draw[thick] (5,7) -- (7,7);
\draw[thick] (8,7) -- (9,7);
\draw[thick] (4,6) -- (5,6);
\draw[thick] (6,6) -- (8,6);
\draw[thick] (5,5) -- (7,5);
\draw[thick] (8,5) -- (9,5);
\draw[thick] (5,4) -- (7,4);
\draw[thick] (8,4) -- (9,4);
\draw[thick] (6,3) -- (9,3);

\draw[thick] (3,9) -- (3,6);
\draw[thick] (4,9) -- (4,7);
\draw[thick] (4,6) -- (4,5);
\draw[thick] (5,8) -- (5,6);
\draw[thick] (5,5) -- (5,4);
\draw[thick] (6,8) -- (6,7);
\draw[thick] (6,6) -- (6,4);
\draw[thick] (7,9) -- (7,7);
\draw[thick] (7,6) -- (7,4);
\draw[thick] (7,3) -- (7,2);
\draw[thick] (8,9) -- (8,6);
\draw[thick] (8,5) -- (8,1);

\filldraw[fill=black] (0,9) circle (0.08cm);
\filldraw[fill=black] (1,9) circle (0.08cm);
\filldraw[fill=black] (2,9) circle (0.08cm);
\filldraw[fill=black] (3,9) circle (0.08cm);
\filldraw[fill=black] (4,9) circle (0.08cm);
\filldraw[fill=black] (5,9) circle (0.08cm);
\filldraw[fill=black] (6,9) circle (0.08cm);
\filldraw[fill=black] (7,9) circle (0.08cm);
\filldraw[fill=black] (8,9) circle (0.08cm);
\filldraw[fill=black] (9,8) circle (0.08cm);
\filldraw[fill=black] (9,7) circle (0.08cm);
\filldraw[fill=black] (9,6) circle (0.08cm);
\filldraw[fill=black] (9,5) circle (0.08cm);
\filldraw[fill=black] (9,4) circle (0.08cm);
\filldraw[fill=black] (9,3) circle (0.08cm);
\filldraw[fill=black] (9,2) circle (0.08cm);
\filldraw[fill=black] (9,1) circle (0.08cm);

\end{tikzpicture}

\caption{A nonregular subdivision of $9 \Delta_2$} \label{NonCoherentSubdiv}
\end{figure}
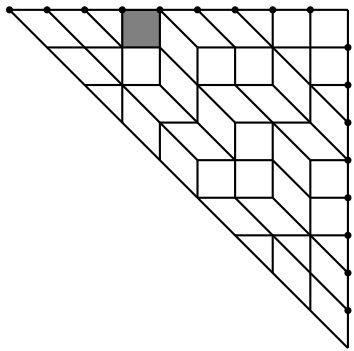

We start with the mixed 
subdivision of $9 \Delta_2$ shown in \cref{NonCoherentSubdiv}.
The subdivision of the central hexagon (with each side of length $3$)
is a standard example of a nonregular  subdivision of a hexagon into rhombi,
originally found by Richter-Gebert,
see \cite[Figure 9]{EdelmanReiner}.
Thus, this mixed subdivision of $9 \Delta_2$ is not regular.

Mixed subdivisions of $b \Delta_{a-1}$ are 
dual to arrangements of $b$ labeled tropical pseudohyperplanes in $\TT \PP^{a-1}$.
The arrangement of $9$ tropical pseudohyperplanes in $\TT \PP^2$ which 
is dual to the mixed subdivision from \cref{NonCoherentSubdiv} is 
shown in \cref{NonCoherentHyperplanes}.
In this figure we have labeled 
the coordinates of $\TT \PP^2$ by $\{ 4, 8, 12 \}$ -- 
placing the labels  at the ``ends" of the rays, according to which coordinate is becoming large along the ray -- and 
labelled the tropical pseudohyperplanes by $\{ 1,2,3,5,6,7,9,10,11 \}$, 
placing the label at the trivalent point.

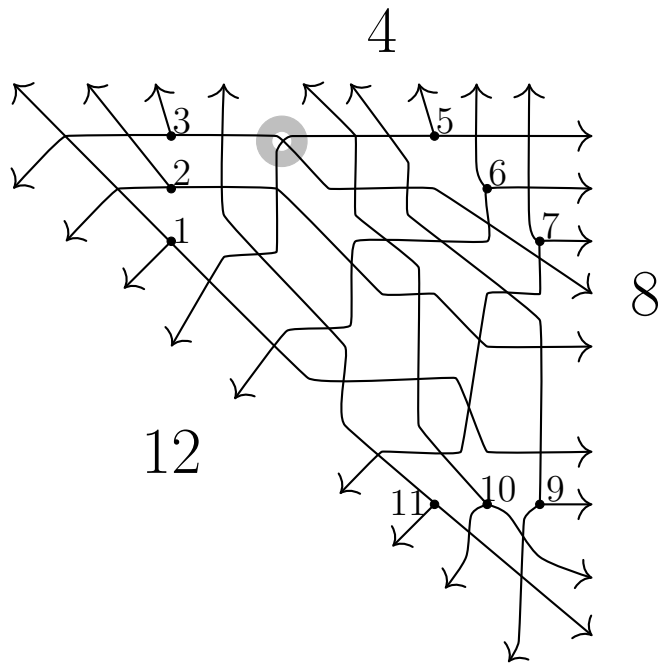
\begin{figure}
	\begin{tikzpicture}[scale=0.7]

\draw[fill={rgb:black,1;white,3},draw={rgb:black,1;white,3}] (4.1,8.9) circle (0.48cm);
\draw[fill=white,draw=white] (4.1,8.9) circle (0.17cm);

\node at (6,11) {\huge 4};
\node at (11,6) {\huge 8};
\node at (2,3) {\huge 12};

\node at (2.2,7.2) {\large 1};
\filldraw[fill=black] (2,7) circle (0.08cm); 
\node at (2.2,8.3) {\large 2};
\filldraw[fill=black] (2,8) circle (0.08cm); 
\node at (2.2,9.3) {\large 3};
\filldraw[fill=black] (2,9) circle (0.08cm);
\node at (7.2,9.3) {\large 5};
\filldraw[fill=black] (7,9) circle (0.08cm); 
\node at (8.2,8.3) {\large 6};
\filldraw[fill=black] (8,8) circle (0.08cm); 
\node at (9.2,7.3) {\large 7};
\filldraw[fill=black] (9,7) circle (0.08cm);
\node at (9.3,2.3) {\large 9};
\filldraw[fill=black] (9,2) circle (0.08cm); 
\node at (8.2,2.3) {\large 10};
\filldraw[fill=black] (8,2) circle (0.08cm); 
\node at (6.5,2.03) {\large 11};
\filldraw[fill=black] (7,2) circle (0.08cm); 

\draw[{<[length=2mm]}-{>[length=2mm]}, thick] plot [smooth, tension=0.1] coordinates {
    (-1,8) (0,9) (4,9) (5,8) (7,8) (10,6)};
\draw[{<[length=2mm]}-{>[length=2mm]}, thick] plot [smooth, tension=0.1] coordinates {
    (0,7) (1,8) (4,8) (6,6) (7,6) (8,5) (10,5)};
\draw[{<[length=2mm]}-{>[length=2mm]}, thick] plot [smooth, tension=0.1] coordinates {
    (-1,10) (4.6,4.4) (7.4,4.4) (8,3) (10,3)};
\draw[{<[length=2mm]}-{>[length=2mm]}, thick] plot [smooth, tension=0.1] coordinates {
    (2,5) (3,6.7) (4,6.8) (4,8.7) (4.3,9) (10,9)};
\draw[{<[length=2mm]}-{>[length=2mm]}, thick] plot [smooth, tension=0.2] coordinates {
    (3.2,4) (4.2,5.3) (5.4,5.4) (5.5,7) (8,7) (8,8) (10,8)};
\draw[{<[length=2mm]}-{>[length=2mm]}, thick] plot [smooth, tension=0.15] coordinates {
    (3,10) (3,7.5) (5.3,5) (5.3,3.5) (10,-0.5)};
\draw[{<[length=2mm]}-{>[length=2mm]}, thick] plot [smooth, tension=0.1] coordinates {
    (5.2,2.2) (6,3) (7.5,3) (8,6) (9,6) (9,7) (10,7)};
\draw[{<[length=2mm]}-, thick] plot [smooth, tension=0.1] coordinates {
    (4.5,10) (5.5,9) (5.5,7.5) (6.7,6.5) (6.7,3.5) (8,2)};
\draw[{<[length=2mm]}-, thick] plot [smooth, tension=0.1] coordinates {
    (5.4,10) (6.5,8.5) (6.5,7.5) (9,5.5) (9,2)};
\draw[-{>[length=2mm]}, thick] plot [smooth, tension=0.1] coordinates { 
    (9,2) (10,2)};
\draw[-{>[length=2mm]}, thick] plot [smooth, tension=0.2] coordinates { 
    (9,2) (8.7,1.7) (8.6,-0.6) (8.4,-1)};
\draw[-{>[length=2mm]}, thick] plot [smooth, tension=0.5] coordinates { 
    (8,2) (8.4,1.8) (9,1) (10,0.6)};
\draw[-{>[length=2mm]}, thick] plot [smooth, tension=0.5] coordinates { 
    (8,2) (7.7,1.8) (7.6,1) (7.2,0.4)};
\draw[-{>[length=2mm]}, thick] plot [smooth, tension=0.1] coordinates { 
    (7,2) (6.2,1.2)};
\draw[-{>[length=2mm]}, thick] plot [smooth, tension=0.5] coordinates { 
    (9,7) (8.8,7.4) (8.8,10)};
\draw[-{>[length=2mm]}, thick] plot [smooth, tension=0.5] coordinates { 
    (8,8) (7.8,8.4) (7.8,10)};
\draw[-{>[length=2mm]}, thick] plot [smooth, tension=0.5] coordinates { 
    (7,9) (6.7,10)};
\draw[-{>[length=2mm]}, thick] plot [smooth, tension=0.5] coordinates { 
    (2,9) (1.7,10)};
\draw[-{>[length=2mm]}, thick] plot [smooth, tension=0.5] coordinates { 
    (2,8) (0.4,10)};
\draw[-{>[length=2mm]}, thick] plot [smooth, tension=0.5] coordinates { 
    (2,7) (1.1,6.1)};

\end{tikzpicture}

\caption{The dual arrangement of $9$ tropical pseudohyperplanes} \label{NonCoherentHyperplanes}
\end{figure}

Also, by the ``Cayley trick" \cite{Cayley1, Cayley2}, 
mixed subdivisions of $b \Delta_{a-1}$ correspond to 
polyhedral subdivisions of $\Delta_{a-1} \times \Delta_{b-1}$,
with regular mixed subdivisions of $b \Delta_{a-1}$ corresponding to 
regular polyhedral subdivisions of $\Delta_{a-1} \times \Delta_{b-1}$.
Therefore the mixed subdivision from \cref{NonCoherentSubdiv}
corresponds to a nonregular polyhedral
subdivision of 
 $\Poly_{\{ 4,8,12 \}} \subset \Delta_{3,12}$.

It remains to check that this subdivision is positroidal.
We need to check that each of the $45$ two-dimensional polytopes in 
\cref{NonCoherentSubdiv}, or equivalently, each of the 
$45$ zero-dimensional cells of the tropical
pseudohyperplane arrangement in \cref{NonCoherentHyperplanes},
corresponds to a positroid. 
Letting $x$ be one of these zero dimensional cells, we must check that $G(X(x))$ is a tree in each case, which can be embedded in a disk as in \cref{PlanarTree}. 

For example, let $x$ be the crossing which is circled in \cref{NonCoherentHyperplanes}; the dual rhombus is shaded in \cref{NonCoherentSubdiv}. 
We have
\[ \begin{array}{lcl@{\quad}lcl@{\quad}lcl}
S(H_1, x) &=& \{ 12 \} & S(H_2,x) &=& \{ 12 \} & S(H_3, x) &=& \{ 4, 12 \} \\
S(H_5, x) &=& \{ 4,8 \} & S(H_6, x) &=& \{ 8 \} & S(H_7, x) &=& \{ 8 \} \\
S(H_9, x) &=& \{ 8 \} & S(H_{10}, x) &=& \{ 8 \} & S(H_{11}, x) &=& \{ 12 \} \\
\end{array} \] 
We draw the corresponding tree in Figure~\ref{PlanarTreeEG}.

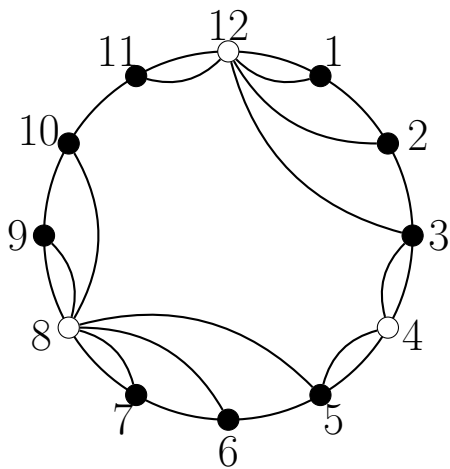
\begin{figure}

	\begin{tikzpicture}[scale=0.7]
\filldraw[fill=white, thick] (0,0) circle (3.5cm);

\node[circle, fill=black, draw=black, inner sep=0pt, minimum size=8pt] (1) at (1.75,3.031) {};
\node[circle, fill=black, draw=black, inner sep=0pt, minimum size=8pt] (2) at (3.031,1.75) {};
\node[circle, fill=black, draw=black, inner sep=0pt, minimum size=8pt] (3) at (3.5,0) {};
\node[circle, fill=white, draw=black, inner sep=0pt, minimum size=8pt] (4) at (3.031,-1.75) {};
\node[circle, fill=black, draw=black, inner sep=0pt, minimum size=8pt] (5) at (1.75,-3.031) {};
\node[circle, fill=black, draw=black, inner sep=0pt, minimum size=8pt] (6) at (0,-3.5) {};
\node[circle, fill=black, draw=black, inner sep=0pt, minimum size=8pt] (7) at (-1.75,-3.031) {};
\node[circle, fill=white, draw=black, inner sep=0pt, minimum size=8pt] (8) at (-3.031,-1.75) {};
\node[circle, fill=black, draw=black, inner sep=0pt, minimum size=8pt] (9) at (-3.5,0) {};
\node[circle, fill=black, draw=black, inner sep=0pt, minimum size=8pt] (10) at (-3.031,1.75) {};
\node[circle, fill=black, draw=black, inner sep=0pt, minimum size=8pt] (11) at (-1.75,3.031) {};
\node[circle, fill=white, draw=black, inner sep=0pt, minimum size=8pt] (12) at (0,3.5) {};

\draw[thick] (12) to[bend right] (1);
\draw[thick] (12) to[bend right] (2);
\draw[thick] (12) to[bend right] (3);
\draw[thick] (12) to[bend left] (11);
\draw[thick] (3) to[bend right] (4);
\draw[thick] (4) to[bend right] (5);
\draw[thick] (5) to[bend right] (8);
\draw[thick] (6) to[bend right] (8);
\draw[thick] (7) to[bend right] (8);
\draw[thick] (8) to[bend right] (10);
\draw[thick] (8) to[bend right] (9);

\node at (2,3.5) {\Large 1};
\node at (3.6,1.9) {\Large 2};
\node at (4,0) {\Large 3};
\node at (3.5,-1.9) {\Large 4};
\node at (2,-3.5) {\Large 5};
\node at (0,-4.1) {\Large 6};
\node at (-2,-3.5) {\Large 7};
\node at (-3.55,-1.9) {\Large 8};
\node at (-4,0) {\Large 9};
\node at (-3.6,2) {\Large 10};
\node at (-2.1,3.5) {\Large 11};
\node at (0,4) {\Large 12};

\end{tikzpicture}
\caption{The planar tree corresponding to the marked point in Figure~\ref{NonCoherentHyperplanes}. Elements of $I$ are shown in white.} \label{PlanarTreeEG}
\end{figure}

\section{Appendix. Combinatorics of cells of the positive Grassmannian.}\label{app}

In \cite{postnikov}, Postnikov defined several families of combinatorial objects which are in bijection with cells of the positive Grassmannian, including \emph{decorated permutations}, 
and equivalence classes of \emph{reduced plabic graphs}. 
Here we review these objects as well as parameterizations of cells.

\begin{defn}\label{defn:decperm}
	A \emph{decorated permutation} of $[n]$ is a bijection $\pi : [n] \to [n]$ whose fixed points are each colored either black (loop) or white (coloop). We denote a black fixed point $i$ by $\pi(i) = \underline{i}$, and a white fixed point $i$ by $\pi(i) = \overline{i}$.
An \emph{anti-excedance} of the decorated permutation $\pi$ is an element $i \in [n]$ such that either $\pi^{-1}(i) > i$ or $\pi(i)=\overline{i}$.
\end{defn}

For example, $\pi = (3,\underline{2},5,1,6,8,\overline{7},4)$ has 
a loop in position $2$, and a coloop in position $7$.
It has three anti-excedances, in positions $4, 7, 8$.
We let $k(\pi)$ denote the number of anti-excedances of $\pi$.





Postnikov showed that the positroids for $Gr_{k,n}^{\ge 0}$ are indexed by decorated permutations of $[n]$ with exactly $k$ anti-excedances 
\cite[Section 16]{postnikov}.

\begin{defn}
A {\it plabic graph}\footnote{``Plabic'' stands for {\itshape planar bi-colored}.}  is an undirected planar graph $G$ drawn inside a disk
(considered modulo homotopy)
with $n$ {\it boundary vertices} on the boundary of the disk,
labeled $1,\dots,n$ in clockwise order, as well as some {\it internal vertices}. Each boundary vertex is incident to a single edge, and each internal vertex is colored either black or white. If a boundary vertex is incident to a leaf (a vertex of degree $1$), we refer to that leaf as a \emph{lollipop}.
\end{defn}

\begin{definition}
A {\it perfect orientation\/} $\O$ of a plabic graph $G$ is a
choice of orientation of each of its edges such that each
black internal vertex $u$ is incident to exactly one edge
directed away from $u$; and each white internal vertex $v$ is incident
to exactly one edge directed towards $v$.
A plabic graph is called {\it perfectly orientable\/} if it admits a perfect orientation.
Let $G_\O$ denote the directed graph associated with a perfect orientation $\O$ of $G$.
The {\it source set\/} $I_\O \subset [n]$ of a perfect orientation $\O$ is the set of $i$ which 
	are sources
 of the directed graph $G_\O$. Similarly, if $j \in \overline{I}_{\O} := [n] - I_{\O}$, then $j$ is a sink of $\O$.
\end{definition}

	See \cref{G25-orientation} for an example.


	\begin{figure}[h]
\centering
\includegraphics[height=1.5in]{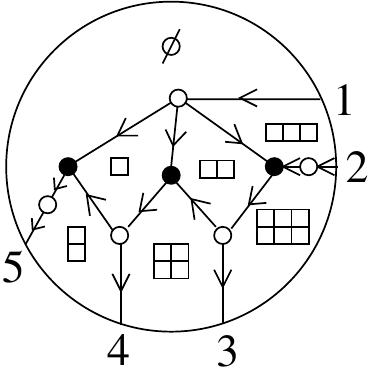}
\caption{
A plabic graph $G$ with trip permutation 
	$(3,4,5,1,2)$, together with 
	a perfect orientation $\O$ 
	with source set
$I_\O = \{1,2\}$.}
\label{G25-orientation}
\end{figure}

All perfect orientations of a fixed plabic graph
$G$ have source sets of the same size $k$, where
$k-(n-k) = \sum \mathrm{color}(v)\cdot(\deg(v)-2)$.
Here the sum is over all internal vertices $v$, $\mathrm{color}(v) = 1$ for a black vertex $v$,
and $\mathrm{color}(v) = -1$ for a white vertex;
see~\cite{postnikov}.  In this case we say that $G$ is of {\it type\/} $(k,n)$.

As shown in \cite[Section 11]{postnikov},
every perfectly orientable plabic graph gives rise to a positroid as follows.
(Moreover, every positroid can be realized in this way.)
\begin{proposition}\label{prop:perf}
Let $G$ be a plabic graph of type $(k,n)$.
Then we have a positroid
$M_G$ on $[n]$ whose bases are precisely
\[\{I_{\O} \mid \O \text{ is a perfect orientation of }G\},\]
where $I_\O$ is the set of sources of $\O$.
\end{proposition}





Each positroid cell corresponds to a family of \emph{reduced plabic graphs}
which are related to each other by 
certain \emph{moves}; see \cite[Section 12]{postnikov}.
From a reduced plabic graph $G$, we can read off the corresponding decorated permutation $\pi_G$ as follows.
\begin{defn}\label{def:rules}
	Let $G$ be a reduced plabic graph of type $(k,n)$ 
	with boundary vertices $1,\dots, n$. For each boundary vertex $i\in [n]$, we follow a path along the edges of $G$ starting at $i$, turning (maximally) right at every internal black vertex, and (maximally) left at every internal white vertex. This path ends at some boundary vertex $\pi(i)$. By \cite[Section 13]{postnikov}, the fact that $G$ is reduced implies that each fixed point of $\pi$ is attached to a lollipop; we color each fixed point by the color of its lollipop. In this way we obtain the \emph{decorated permutation} $\pi_G = \pi$ of $G$. 
	The decorated permutation $\pi_G$ will have precisely $k$ anti-excedances.
\end{defn}

We now explain how to parameterize elements of positroid cells using perfect orientations
of reduced plabic graphs.

We will associate a parameter $x_{\mu}$ to each face of $G$, letting
$\mathcal{P}_G$ denote the indexing set for the faces.
We require that the product $\prod_{\mu\in \mathcal{P}_G} x_{\mu}$ of all parameters equals $1$.
A \emph{flow} $F$ from $I_{\O}$ to a set $J$ of boundary vertices
with
$|J|=|I_{\O}|$
is a collection of paths and closed cycles
in $\O$, all pairwise vertex-disjoint,
such that the sources of the paths are $I_{\O} - (I_{\O} \cap J)$
and the destinations of the paths are $J - (I_{\O} \cap J)$.

Note that each directed path and cycle %
$w$ in $\O$ partitions the faces of $G$ into those
which are on the left and those which are on the right of $w$.
We define the \emph{weight} $\wt(w)$ of
each such path or cycle %
to be the product of
parameters $x_{\mu}$, where $\mu$ ranges over all face labels
to the left of the path.  And we define the
\emph{weight} $\wt(F)$ of a flow $F$ to be the product of the weights of
all paths and cycles in the flow.

Fix a perfect orientation $\O$ of a reduced plabic graph $G$.
Given $J \in {[n] \choose k}$,
we define the {\it flow polynomial}
\begin{equation}\label{eq:Plucker}
p_J^G = \sum_F \wt(F),
\end{equation}
where $F$ ranges over all flows from $I_{\O}$ to $J$.

\begin{example}
Consider the graph
from Figure \ref{G25-orientation}. There are two flows $F$
from $I_{\O}$ to $\{2,4\}$, and
$P^G_{\{2,4\}} = x_{\ydiagram{3}} x_{\ydiagram{2,2}} x_{\ydiagram{3,3}}
+x_{\ydiagram{2}} x_{\ydiagram{3}} x_{\ydiagram{2,2}} x_{\ydiagram{3,3}}$.
There is one flow from $I_{\O}$ to $\{3,4\}$, and
$P^G_{\{3,4\}} = x_{\ydiagram{2}} x_{\ydiagram{3}} x_{\ydiagram{2,2}}
x_{\ydiagram{3,3}}^2.$
\end{example}

The following result is a combination of 
\cite[Theorem 12.7]{postnikov} and 
\cite[Theorem 1.1]{Talaska}.

\begin{theorem}\label{network_param}
	Let $G$ be a reduced plabic graph of type $(k,n)$, and choose
a perfect orientation $\O$ with source set $I_{\O}$.
 Then the map $\Phi_G$ sending
$(x_{\mu})_{\mu \in \mathcal{P}_G}
	\in (\R_{>0})^{\mathcal{P}_G}$ to the collection 
	of flow polynomials $\{p_J^G\}_{J\in {[n] \choose k}}$
	is a homemorphism from 
	$(\R_{>0})^{\mathcal{P}_G}$ to 
	the corresponding positroid cell $S_G \subset Gr_{k,n}$ (realized in its Pl\"ucker 
	embedding).
\end{theorem}

\bibliographystyle{alpha}
\bibliography{bibliography}

\end{document}